\documentclass[10pt,a4paper,reqno]{amsart}
\usepackage[T1]{fontenc}
\usepackage[latin1]{inputenc}
\usepackage{amsmath}
\usepackage{amsthm}
\usepackage{amsfonts}
\usepackage{amssymb}
\usepackage{graphicx}
\usepackage{amsbsy}
\usepackage{mathrsfs}
\usepackage{bbm}
\usepackage{dsfont}
\usepackage{eucal}
\usepackage{bm}
\newcommand{\be}{\begin{equation}}
\newcommand{\ee}{\end{equation}}
\newcommand{\ba}{\begin{array}}
\newcommand{\ea}{\end{array}}
\newcommand{\bea}{\begin{eqnarray}}
\newcommand{\eea}{\end{eqnarray}}
\newcommand{\bee}{\begin{eqnarray*}}
\newcommand{\eee}{\end{eqnarray*}}
\newcommand{\tr}{{\rm Tr}}
\newcommand{\mass}{{\bf M}}

\newtheorem{thm}{Theorem}
\newtheorem{lemma}{Lemma}

\newtheorem{remark}{Remark}
\newtheorem{conjecture}{Conjecture}
\newtheorem{claim}{Claim}
\newtheorem*{remark*}{Remark}

\numberwithin{equation}{section}

\catcode`@=11
\def\section{\@startsection{section}{1}%
  \z@{1.5\linespacing\@plus\linespacing}{.5\linespacing}%
  {\normalfont\bfseries\large\centering}}
\catcode`@=12

\newcommand{\R}{\mathbb{R}}
\newcommand{\Z}{\mathbb{Z}}
\newcommand{\C}{\mathbb{C}}
\newcommand{\N}{\mathbb{N}}

\newcommand{\pt}{\partial}
\renewcommand{\leq}{\leqslant}
\renewcommand{\geq}{\geqslant}

\newcommand{\Dcou}{\mathcal{D}}
\newcommand{\weakto}{\rightharpoonup}
\newcommand{\wto}{\weakto}
\newcommand{\ii}{\infty}
\newcommand\1{{\ensuremath {\mathds 1} }}
\newcommand{\eps}{\varepsilon}
\newcommand{\cM}{\mathcal M}
\newcommand{\cE}{\mathcal E}
\newcommand{\Mcrit}{M_{\rm c}}
\newcommand\pscal[1]{{\ensuremath{\left\langle #1 \right\rangle}}}
\newcommand{\norm}[1]{ \left| \! \left| #1 \right| \! \right| }

\usepackage[active]{srcltx}

\begin{document}

\title[On Singularity formation for the $L^2$-critical Boson star equation]{On singularity formation for \\ the $L^2$-critical Boson star equation}

\author[E. Lenzmann]{Enno LENZMANN}
\address{Department of Mathematical Sciences, University of Copenhagen, Universitetspark 5, 2100 Copenhagen \O, Denmark.}
\email{lenzmann@math.ku.dk}

\author[M. Lewin]{Mathieu LEWIN}
\address{CNRS \& Laboratoire de Mathématiques (UMR 8088), Université de Cergy-Pontoise, F-95000 Cergy-Pontoise, France.}
\email{Mathieu.Lewin@math.cnrs.fr}

\begin{abstract}
We prove a general, non-perturbative result about finite-time blowup solutions for the $L^2$-critical boson star equation
$$
i \partial_t u = \sqrt{-\Delta+m^2} \, u - ( |x|^{-1} \ast |u|^2  ) u,
$$
in $d=3$ space dimensions. Under the sole assumption that $u=u(t,x)$ blows up in the energy space $H^{1/2}$ at finite time $0 < T < +\infty$, we show that $u(t,\cdot)$ has a unique weak limit in $L^2$ and that $|u(t,\cdot)|^2$ has a unique weak limit in the sense of measures as $t \to T^-$.  Moreover, we prove that the limiting measure exhibits minimal mass concentration. A central ingredient used in the proof is a ``finite speed of propagation'' property, which puts a strong rigidity on the blowup behavior of $u=u(t,x)$. 

As the second main result, we prove that any radial finite-time blowup solution $u=u(t,|x|)$ converges {\em strongly} in $L^2(\{ |x| \geq R \})$ as $t \to T^-$ for any $R > 0$. For radial solutions, this result establishes a large data blowup conjecture for the $L^2$-critical boson star equation.

We also discuss some extensions of our results to other $L^2$-critical theories of gravitational collapse, in particular to critical Hartree-type equations.
\end{abstract}

\maketitle

\section{Introduction and Main Results}

In this paper, we are interested in deriving general, non-perturbative results on singularity formation (blowup) for the $L^2$-critical boson star equation.  The corresponding Cauchy problem is given by
\begin{equation} \label{eq:pde}
\left \{ \begin{array}{ll} i\, \pt_t u = \sqrt{-\Delta+m^2} \, u - ( |x|^{-1} \ast |u|^2  ) u,  & (t,x) \in [0,T) \times \R^3, \\[1ex]
u(0,x) = u_0(x), \quad u_0 : \R^3 \to \C. \end{array} \right .
\end{equation}
Here $m \geq 0$ is a physical constant, the operator $\sqrt{-\Delta+m^2}$ is defined via its symbol $\sqrt{\xi^2 + m^2}$ in Fourier space, and $\ast$ denotes convolution of functions on $\R^3$. With regard to physical applications, we mention that the nonlinear evolution problem \eqref{eq:pde} plays a central role in the theory of the gravitational collapse of boson stars; see the recent works \cite{ElSc07,ElMi10} for the rigorous derivation of \eqref{eq:pde} from many-body quantum mechanics; see also the seminal work \cite{LiTh84,LiYa87} in the time-independent setting.

Let us briefly recall the known facts concerning the well-posedness and blowup for the $L^2$-critical boson star equation~\eqref{eq:pde}. From \cite{Le07} we have that the Cauchy problem \eqref{eq:pde} is locally well-posed for initial data $u_0$ in the Sobolev space $H^{s}(\R^3)$ for any $s \geq 1/2$. In particular, we have the fundamental blowup criterion that if $u=u(t,x)$ cannot be extended to all times $t \geq 0$, then the solution exhibits finite-time blowup in $H^{1/2}(\R^3)$ meaning that
\begin{equation} \label{eq:blowup}
\| u(t,\cdot) \|_{H^{1/2}} \to +\infty \quad \mbox{as} \quad t \to T^-,
\end{equation}
for some $0 < T < +\infty$. Moreover, along its time interval of existence, the solution $u=u(t,x)$ satisfies the conservation laws of energy and $L^2$-mass which are given by
\begin{equation}
 E[u] = \frac{1}{2} \int_{\R^3} \overline{u} \sqrt{-\Delta+m^2} \, u - \frac{1}{4} \int_{\R^3} ( |x|^{-1} \ast |u|^2 ) |u|^2 ,
\label{eq:def_energy}
\end{equation}
\begin{equation}
 M[u] = \int_{\R^3} |u|^2.
\end{equation}
Similar to $L^2$-critical NLS (see, e.\,g., \cite{We86}), we can combine the conservation laws for $E[u]$ and $M[u]$ with an interpolation inequality to obtain the following sufficient condition for global well-posedness in energy space: If the initial datum $u_0 \in H^{1/2}(\R^3)$ satisfies the smallness condition
\begin{equation} \label{eq:glo}
M[u_0] < \Mcrit ,
\end{equation}   
then the corresponding solution $u=u(t,x)$ of \eqref{eq:pde} extends to all times $t \geq 0$. Here $\Mcrit > 0$ is the so-called {\em ``critical (or minimal) mass''}  and it is given by
\begin{equation}
\Mcrit = \| Q \|_{L^2}^2 
\end{equation}  
where $Q =Q(|x|)  \in H^{1/2}(\R^3)$ is a radial ground state solution of the nonlinear equation 
\begin{equation} 
\sqrt{-\Delta} \, Q + Q - \big ( |x|^{-1} \ast |Q|^2 \big ) Q = 0 ,
\label{eq:def_Q}
\end{equation}
see Appendix \ref{app:Mcrit} for more details. Equivalently, we can define $\Mcrit> 0$ as the unique positive number such that 
\begin{equation}
\inf_{\substack{u\in H^{1/2}(\R^3)\\ M[u]=\lambda}}E[u]\;\begin{cases}
\geq0&\text{if $\lambda\leq\Mcrit$,}\\
=-\ii & \text{if $\lambda>\Mcrit$.}
\end{cases}
\label{eq:energy_critical_mass} 
\end{equation}
Note that the global well-posedness criterion \eqref{eq:glo} is independent of the physical constant $m \geq 0$. 

Regarding the existence of finite-time blowup solutions for \eqref{eq:pde}, we mention the following criterion that was proven in \cite{FrLe07}: If the initial datum $u_0 = u_0(|x|) \in C^\infty_0(\R^3)$ is radially symmetric and has negative energy
 \begin{equation} \label{ineq:Eneg}
E[u_0] < 0,
\end{equation}
then the corresponding solution $u=u(t,x)$ blows up in finite time. By \eqref{eq:energy_critical_mass}, we note that condition \eqref{ineq:Eneg} implies that $u_0$ satisfies the strict inequality $M[u_0] > \Mcrit$. Furthermore, a closer inspection of the arguments in \cite{FrLe07} reveals that the global well-posedness criterion \eqref{eq:glo} is almost optimal in the sense that, for any $\eps > 0$, there exists a finite-time (radial) blowup solution with  $M[u_0] = \Mcrit + \eps$. In particular, this blowup result rigorously justifies (for radial data) the physical intuition that $\Mcrit$ describes the {\em ``Chandrasekhar limiting mass''} of relativistic boson stars. However, we mention that the existence of {\em non-radially symmetric} blowup solutions for the evolution problem \eqref{eq:pde} still presents a major open problem.

\subsection{Nonperturbative description of blowup for the boson star equation~\eqref{eq:pde}}
In the present paper, we derive general properties of finite-time blowup solutions for the $L^2$-critical boson star equation~\eqref{eq:pde}. Our first main result  deals with blowup solutions without any symmetry assumption.
\begin{thm}[\bf General Case]\label{thm:nonrad}

Suppose that $u \in C^0 ([0,T), H^{1/2}(\R^3))$ solves \eqref{eq:pde} with $m\geq0$, and blows up at finite time $0 < T < +\infty$. Then the following holds.
\begin{enumerate}

\smallskip

\item[$(i)$] {\bf Existence and uniqueness of weak $L^2$-limit:} There exists a unique function $u_* \in L^2(\R^3)$ such that 
$$\mbox{$u(t,\cdot) \weakto u_*$ weakly in $L^2(\R^3)$ as $t \to T^-$}.$$ 
Moreover, we have that $u(t) \to u_*$ strongly in $H^{-1/2}(\R^3)$ as $t \to T^-$.

\smallskip

\item[$(ii)$] {\bf Existence and uniqueness of blowup measure:} There exists a unique, positive, regular Borel measure $\mu \in \mathcal{M}(\R^3)$ such that 
$$
\mbox{$|u(t,\cdot)|^2 \weakto \mu$ weakly in $\mathcal{M}(\R^3)$ as $t \to T^-$},
$$
and we have that $\mu(\R^3) = \int_{\R^3} |u_0|^2$.  

\smallskip

\item[$(iii)$] {\bf Minimal mass concentration point:} There is some point $x_*  \in \R^3$ such that 
$$
\mu(\{ x_* \}) \geq \Mcrit,
$$
where $\mu \in \mathcal{M}(\R^3)$ is given in $(ii)$. 
\end{enumerate}
\end{thm}

\noindent
Let us emphasize that our result covers the zero-mass case $m=0$. We now make the following comments about this result.

\subsubsection*{Improvement for Radial Data}
As our second main result, we will obtain a considerable strengthening of Theorem 1 for radially symmetric blowup solutions, see Theorem \ref{thm:rad} below.

\subsubsection*{Generalizations}
The arguments to prove Theorem \ref{thm:nonrad} can be readily extended to the following generalization of \eqref{eq:pde} given by 
\begin{equation} \label{eq:pde2}
i \partial_t u = L u + (\Phi \ast |u|^2) u \quad \mbox{in $\R^d$ with $d \geq 3$}. 
\end{equation}
Here $L$ is a self-adjoint pseudo-differential operator of order 1 satisfying suitable properties. For instance, it is sufficient to assume that $L$ if of the form $L=L(i\nabla)$ with real-valued symbol $\widehat{L}(\xi)$ such that 
\begin{equation}
\nabla_{\xi}\widehat{L} \in L^\infty, \quad  A |\xi| \leq \widehat{L}(\xi) \leq B |\xi| \quad \mbox{for $|\xi| \geq C$}, 
\end{equation}
with some constants $A, B, C > 0$. Moreover, the convolution kernel $\Phi = \Phi(x)$ is supposed to be of the form 
\begin{equation}
 \Phi(|x|) = - \frac{1}{|x|}+ w(x),
\end{equation}
where $w \in L^\infty(\R^3)$ is a real-valued and even function satisfying the decay estimate $|w(x)| \leq C(1+|x|)^{-1}$ for $x \in \R^d$, with some constant $C >0$.

Another (and physically important) generalization of \eqref{eq:pde} are {\em Hartree equations,} which we discuss in some detail at the end of this section. See also Theorem \ref{thm:Hartree} below. 

\subsubsection*{Minimal Mass Concentration}
From $(ii)$ and $(iii)$ in Theorem \ref{thm:nonrad} above, we conclude that if $x(t) \to x_*$ as $t \to T$, then
\begin{equation}
\forall  R > 0  : \quad \liminf_{t \to T^-} \int_{|x-x(t)| \leq R} |u(t,x)|^2  \geq \Mcrit .
\end{equation}
Such a phenomenon of {\em ``minimal mass concentration''} was already proved for \eqref{eq:pde} for radially symmetric solutions in \cite{FrLe07}. It is also known to hold for $L^2$-critical nonlinear Schrödinger equation (see \cite{MeTs90,Na90,We86,HmKe05}), where $x=x(t)$ is some appropriate function. In this latter case, proving that $x(t)$ tends to some finite limit $x_* \in \R^3$ as $t \to T$ is much more complicated. This has only been resolved so far for radially symmetric solutions (see, e.\,g., \cite{MeTs90}) and for blowup solutions with $L^2$-mass close to the critical mass; see \cite{MeRa05} and references therein.

\subsubsection*{Blowup Conjecture}
It would be highly desirable to describe the singular part of the blowup measure $\mu \in \mathcal{M}(\R^3)$ in more detail. Here, it seems natural to expect that the singular part of $\mu$ is a finite sum of Dirac measures with each having at least mass $\Mcrit$. Furthermore, we expect that $u(t,\cdot)$ converges strongly in $L^2$ outside the singular set  of $\mu$. To formalize this remark, we state the following general blowup conjecture, which is inspired of a famous similar conjecture by Merle and Rapha\"el in \cite{MeRa05} formulated for the $L^2$-critical nonlinear Schrödinger equation (NLS).    
 
\begin{conjecture}
Suppose that $u \in C^0([0,T); H^{1/2}(\R^3))$ solves \eqref{eq:pde} with $m \geq 0$ and blows up at finite time $0 < T < +\infty$. Let $u^*\in L^2(\R^3)$ and $\mu\in \cM(\R^3)$ be as in Theorem~\ref{thm:nonrad}. Then there exist finitely many points $\{ x_1, \ldots, x_L \} \subset \R^3$ with $1 \leq L \leq \frac{\int |u_0|^2}{\Mcrit}$ such that, as $t \to T^-$, the following holds:
$$
\mbox{$u(t,\cdot) \to u_*$ strongly in $L^2 \left ( \R^3 \setminus \bigcup_{1 \leq i \leq L} B(x_i,R) \right)$ for all $R >0$,}
$$
and
$$
|u(t,\cdot)|^2 \weakto \mu = \sum_{1 \leq i \leq L} M_i \delta_{x=x_i} + |u_*|^2\ \text{weakly in $\cM(\R^3)$, with}\ M_i \geq  \Mcrit.
$$
\end{conjecture}

\medskip
\noindent
In fact, we can prove the previous blowup conjecture for radially symmetric solutions. We have the following main result.

\begin{thm}[{\bf Radial Case}] \label{thm:rad}
If $u_0 = u_0(|x|) \in H^{1/2}(\R^3)$ is radially symmetric, then Conjecture 1 holds true. More precisely, with $u_* \in L^2(\R^3)$ and $\mu \in \mathcal{M}(\R^3)$ as in Theorem \ref{thm:nonrad}, we have the following properties. 
\begin{enumerate}
\item[$(i)$] For every $R > 0$, it holds that
$$
\mbox{$u(t,\cdot) \to u_*$ strongly in $L^2\left ( \R^3 \setminus B(0,R) \right)$ as $t \to T^-$.}
$$
\item[$(ii)$] The blowup measure $\mu \in \mathcal{M}(\R^3)$ is given by
$$\mu = M \delta_{x=0} + |u_*|^2, $$ 
with some $M \geq \Mcrit$.
\end{enumerate}
Moreover, if $x \, u_0(|x|) \in H^{1/2}(\R^3)$ then $(i)$ can be replaced by
\begin{enumerate}
\item[$(i')$] For every cutoff function $\chi \in C^\infty(\R^3)$ with $0 \leq \chi \leq 1$ and $\mathrm{supp} \, \chi \subset \R^3 \setminus \{ 0 \}$, it holds that $\chi \, u(t,\cdot) \weakto \chi \, u_*$ weakly in $H^{1/2}(\R^3)$ as $t \to T^-$ and
$$
\mbox{$\chi \, u(t,\cdot) \to \chi \, u_*$ strongly in $H^s(\R^3)$ for $0 \leq s < 1/2$ as $t \to T^-$}.
$$
\end{enumerate}
\end{thm}

\medskip\noindent
We make the following comments about this result.

\subsubsection*{Generalization}
We remark that $(i)$ and $(ii)$ of Theorem \ref{thm:rad} can be readily generalized to the evolution problem \eqref{eq:pde2}, {\em provided} that we consider $d=3$ space dimensions. The reason for this restriction stems from the fact that we make use of Newton's theorem, which only holds true for the convolution kernel $|x|^{-1}$ in dimension $d=3$. See below for more details. 

\subsubsection*{Regularity of $u_*$} 
It would be interesting to better understand the regularity properties of the limiting function $u_*$. Under the sole assumption that the radially symmetric initial datum $u_0$ belongs to $H^{1/2}(\R^3)$, we do not know whether $\chi \, u_*$ is  in $H^{1/2}(\R^3)$ for $\chi$ as in $(i')$. But we expect that $(i')$ also holds true in this case. On the other hand, the regularity of $u_*$ at $x=0$ is unclear to us at the present. 

\subsubsection*{Quantization of Blowup Mass and Universilaty of Blowup Profile}
It seems to be a very interesting (but difficult) open problem to understand whether $M \geq \Mcrit$ in $(ii)$ can only attain discrete values. In particular, it seems natural to expect that $M=\Mcrit$ holds in $(ii)$ above, provided that $\Mcrit < \| u_0 \|_{L^2}^2 < \Mcrit + \alpha$ with some universal small constant $\alpha > 0$. Moreover, in analogy to the fundamental result for $L^2$-critical NLS in \cite{MeRa04}, we might in fact expect that the ground state $Q$ provides (modulo symmetries) the {\em universal blowup profile} for data with $L^2$-mass close to $\Mcrit$; i.\,e., 
$$
\mbox{$e^{i \gamma(t)} u(t, \lambda(t) x +x_0(t)) \to Q$ in $L^2_{\mathrm{loc}}(\R^3)$ as $t \to T^-$},
$$ 
where $\lambda(t) = \frac{\| |\nabla|^{1/2} Q \|_{L^2}}{ \| |\nabla|^{1/2} u(t) \|_{L^2}}$, $\gamma(t) \in \R$ and $x_0(t) \in \R^3$  are  modulation parameters corresponding to scaling, phase and translation symmetry, respectively. Note that for radial solutions $u=u(t,|x|)$ we must have  that $x_0(t) \equiv 0$ holds, by Theorem \ref{thm:rad} above.

\bigskip
Let us conclude this section with some comments on the proof of the main results.

In both the proofs of Theorem 1 and 2, we will use the important fact that equation \eqref{eq:pde} exhibits some version of a \emph{finite speed of propagation}. Properties of the same sort were already exploited in different contexts (for instance when proving the existence of global solutions to the defocusing critical wave equation, see, e.g.,~\cite{Struwe-88,Tao-06}), but, to our knowledge, it was never used for blow-up solutions to the Boson star equation \eqref{eq:pde}.

In our context, \emph{finite speed of propagation} has the effect that the variation of mass in any domain stays bounded. More precisely, we show in Lemma~\ref{lem:finite} below that the derivative
$$\frac{d}{dt}\int_{\R^3}\chi |u(t)|^2$$
is uniformly bounded in time for any fixed function $\chi \in W^{1,\infty}(\R^3)$. This fact is itself a simple consequence of a commutator estimate due to Calder\'on~\cite{Calderon-65,St93} given by
\begin{equation}
\| [\sqrt{-\Delta+m^2}, \chi ] \|_{L^2 \to L^2} \lesssim \| \nabla \chi \|_{L^\infty}.
\label{eq:Stein_intro}
\end{equation} 
The fact that the localized mass $\int \chi |u(t)|^2$ cannot vary too rapidly yields the existence of the limiting blowup measure $\mu$. 

To prove that the limiting measure $\mu$ must contain a delta with mass $\geq\Mcrit$, we borrow ideas already used in the context of NLS, see, e.g.,~\cite{MeTs90,Na90,We86,HmKe05}, which are mainly based on concentration-compactness techniques~\cite{Lions-84a,Lions-84b}. However, the extension to the nonlocal pseudodifferential operator $\sqrt{-\Delta+m^2}$ poses some technical difficulties. For this reason, we give a detailed proof and, for the sake of clarity, we even provide in Appendix~\ref{app:CC_frac_Sobolev} a self-contained description of the possible lack of compactness of bounded sequences in fractional Sobolev spaces. 

As we have already remarked, the proof of Theorem~\ref{thm:nonrad} can be extended to a corresponding blowup result for the generalized problem \eqref{eq:pde2} in $d \geq 3$ space dimensions. In particular, the proof does {\em not} hinge on specific properties of $|x|^{-1}$ being the Green's function of some local differential operator. However,  to deal with the radial case in Theorem~\ref{thm:rad}, we use the fact that the convolution of a radial function with $|x|^{-1}$ in $\R^3$ obeys a pointwise estimate, due to Newton's theorem. As explained in Section~\ref{sec:proof_radial} below, the latter fact guarantees that the potential $|x|^{-1}\ast|u(t)|^2$ is uniformly bounded on $\R^3\setminus B(0,R)$ for any $R>0$, whenever $u=u(t,|x|)$ is radially symmetric in $x \in \R^3$. This fact, combined with \eqref{eq:Stein_intro} again, is itself enough to give strong convergence in $L^2(\R^3\setminus B(0,R))$.

\begin{remark}\em
Let us emphasize that the finite speed of propagation is alone not enough to explain why radial solutions can only blow up at the origin (and not on a ring or a disc). It is also important that the gravitational interaction is nonlocal. For instance, for the focusing wave equation (which, of course, exhibits finite speed of propagation), 
\begin{equation} \label{eq:wave}
\left(\frac{\partial^2}{\partial t^2}-\Delta\right) u - |u|^{p-1} u = 0, \quad \mbox{in $\R^3$ where $p > 1$},
\end{equation}
it is well-known that there exist radial solutions $u=u(t,|x|)$ with finite energy, which blow up on the whole disc ${|x| < T}$ as $t \to T^-$, by a simple ODE-type mechanism.\footnote{For the wave equation \eqref{eq:wave} take the initial conditions $u(0,x) = C T^{-\alpha} \phi(x)$ and $\partial_t u(0,x) = \alpha C T^{\alpha-1} \phi(x)$ with $\alpha = \frac{2}{p-1}$ and $C = \left [ \alpha (\alpha +1) \right ]^{\frac{\alpha}{2}}$, where $\phi \in C^\infty_0(\R^3)$ is radial and satisfies $0 \leq \phi \leq 1$ and $\phi(x) \equiv 1$ for $|x| < 2T$.}
Our Theorem~\ref{thm:rad} shows that this phenomenon is impossible for radial solutions to the boson star equation~\eqref{eq:pde}.
\end{remark}

Our method is certainly rather specific to the boson star equation~\eqref{eq:pde}. But, on the other hand, it is based on natural physical properties of the model under consideration (e.\,g., finite speed of propagation reflecting special relativity, Newton's theorem), which lead to a purely nonperturbative result. In the next section, we detail a further generalization of Theorems \ref{thm:nonrad} and \ref{thm:rad} to Hartree equations, which  arise as a model problem for collapsing fermionic relativistic stars.

\subsection{Hartree theory}
The $L^2$-critical boson star equation~\eqref{eq:pde} is one of the simplest (yet quite rich) physical models describing the dynamical collapse of stars. For fermionic stars like neutron stars or white dwarfs, one has to use more complicated models~\cite{LiTh84,LiYa87,FrLe07b,LeLe10,HaiLenLewSch-10}. In this section we consider Hartree-type theories, in which the state of the system is not described by a single function $u\in L^2(\R^3)$, but rather by a nonnegative and self-adjoint trace-class operator $\gamma$ acting on $L^2(\R^3)$. The time evolution of the state $\gamma(t)$ is determined by the von Neumann equation
\begin{equation}
\left \{ \begin{array}{ll} i \,\pt_t \gamma = \big[H_\gamma, \gamma\big],  & 0\leq t<T, \\[1ex]
\gamma(t=0) = \gamma_0\geq0. \end{array} \right.
\label{eq:HF}
\end{equation}
Here $[A,B] \equiv AB-BA$ denotes the commutator of two (possibly unbounded) operators $A$ and $B$ acting on $L^2(\R^3)$, and $H_\gamma$ is the so-called \emph{mean-field operator} defined by
\begin{equation}
H_\gamma=\sqrt{-\Delta+m^2} -  |x|^{-1} \ast \rho_\gamma,
\label{eq:def_H_gamma}
\end{equation}
where $\rho_\gamma(x) = \gamma(x,x) \in L^1(\R^3)$ is the density associated with the trace-class operator $\gamma$, meaning that $\rho_\gamma$ is the unique function in $L^1(\R^3)$ such that (where $\tr( \cdot)$ denotes the trace of an operator):
$$\forall V\in L^\ii(\R^3):  \quad \tr(\gamma V)=\int_{\R^3}\rho_\gamma(x)\,V(x)\,dx.$$

We note that equation~\eqref{eq:HF} generalizes the boson star equation~\eqref{eq:pde}. Indeed, when $\gamma_0$ is the operator with integral kernel $\gamma_0(x,y)=u_0(x)\overline{u_0(y)}$, then the unique solution to~\eqref{eq:HF} is the operator $\gamma(t)$ with kernel $\gamma(t,x,y)=u(t,x)\overline{u(t,y)}$, where $u(t,x)$ is the solution to~\eqref{eq:pde}. 

It is easy to see (at least formally) that the solution $\gamma(t)$ of \eqref{eq:HF} satisfies $\gamma(t)=U(0,t)^*\,\gamma_0\, U(0,t)$ where $U(0,t)$ is a time-dependent family of unitary operators solving the equation
$$\left\{\begin{array}{ll}
 i\, \pt_t\, U(0,t)=-U(0,t)\, H_{\gamma(t)} & \text{for $0<t<T$}\\
U(t=0)=1,
\end{array}\right.$$
where $1$ denotes the identity operator on $L^2(\R^3)$. In particular, the operator $\gamma(t)$ is unitarily equivalent to the initial state $\gamma_0$ for all times $t\in[0,T)$ and the total mass of the system is conserved 
$$M[\gamma(t)]:=\tr(\gamma)=M[\gamma_0].$$
Also, if $\gamma_0\geq0$ as an operator, then $\gamma(t)\geq0$ for all times $t\in[0,T)$, as we easily verify.

For fermionic stars like neutrons stars and white dwarfs, the state $\gamma$ has to satisfy the additional constraint that $\gamma\leq 1$. By unitary equivalence, note that if $0\leq\gamma_0\leq 1$, then $0\leq\gamma(t)\leq 1$ for all times. Of course, the total mass of the system, $M[\gamma_0]=\tr(\gamma_0)$, can be made arbitrarily large, even with this new constraint. In the fermionic case, one usually considers a more complicated model called \emph{Hartree-Fock}, by adding to the mean-field operator $H_\gamma$ a nonlocal operator depending on $\gamma$ (see for instance \cite{HaiLenLewSch-10}). For the sake of simplicity, we do not consider such Hartree-Fock models here.

In analogy to the boson star equation~\eqref{eq:pde}, there is a conserved energy which, in the present case, reads
$$E_{\rm H}[\gamma]= \frac{1}{2} \tr\sqrt{-\Delta+m^2}\,\gamma -\frac{1}{4}\int_{\R^3}(|x|^{-1}\ast\rho_\gamma)\,\rho_\gamma.$$
Actually, we have $E_{\rm H}[\gamma]\geq E[\sqrt{\rho_\gamma}]$ where the right side is the boson star energy~\eqref{eq:def_energy} of $\sqrt{\rho_\gamma}$, see~\cite[Lemma 2.1]{LeLe10} and \cite[Lemma 7.13]{LieLos-01}. In particular, we have 
$$\inf_{\substack{\gamma\geq0,\; M[\gamma] = \lambda \\ \tr\,\sqrt{-\Delta+m^2}\gamma<+\ii\\}}E_{\rm H}[\gamma]=\inf_{\substack{u\in H^{1/2}(\R^3)\\ M[u]=\lambda}}E[u].$$

Existence of local-in-time solutions to~\eqref{eq:HF}, such that $t\mapsto \gamma(t)$ is continuous for an appropriate Sobolev-like operator norm was proved in the finite rank case in \cite{FrLe07b}, but the extension to the general case can be carried over by using the method of \cite{ChaGla-75,Chadam-76,BovPraFan-74,BovPraFan-76}. As in the case of equation~\eqref{eq:pde}, a blowup criterion tells us that when the maximal time of existence $T$ is finite, then it must hold $\tr\sqrt{-\Delta+m^2}\,\gamma(t)\to+\ii$ as $t\to T^-$. Existence of blowup solutions was proved in \cite{FrLe07b}, assuming that the initial state $\gamma_0$ is radial and finite rank. This result was later strengthened in \cite{HaiSch-09,HaiLenLewSch-10}. 

As before, there is a critical mass $\Mcrit$ which, however, might depend on the constraints that are imposed on $\gamma$. If we work under the assumption that $0\leq \gamma\leq \kappa 1$ (for fermions $\kappa=1$), we can define similarly as in~\eqref{eq:energy_critical_mass} an associated critical mass $0<\Mcrit(\kappa)<\ii$, above which blowup solutions are known to exist. The number $\Mcrit(\kappa)$ is nonincreasing with $\kappa$. It diverges like $\kappa^{-3/2}$ for small $\kappa$ (this follows from results of \cite{LiYa87}) and it equals $\Mcrit$, the critical mass of the boson star equation~\eqref{eq:pde}, for all $\kappa\geq \Mcrit$.

Adapting our method explained in the next section for the boson star equation~\eqref{eq:pde} and using techniques similar to those of \cite{LeLe10}, it is not difficult to prove the following result for the Hartree model~\eqref{eq:HF}.

\begin{thm}[\bf Hartree Case]\label{thm:Hartree}

Suppose that the solution $\gamma(t)$ of the Hartree equation \eqref{eq:HF} with $m\geq0$, has initial datum $\gamma_0 \geq 0$ such that
$$\tr\,\sqrt{-\Delta+m^2}\,\gamma_0< +\ii.$$
If $\gamma(t)$ blows up at finite time $0 < T<+\ii$, then the following holds:
\begin{enumerate}

\smallskip

\item[$(i)$] {\bf Existence and uniqueness of weak limit:} There exists a unique, self-adjoint trace-class operator $\gamma_*\geq0$ on $L^2(\R^3)$ such that 
$$\mbox{$\gamma(t) \weakto \gamma_*$ weakly--$\ast$ in the trace class as $t \to T^-$},$$
meaning that $\tr(\gamma(t) K) \to \tr(\gamma_* K)$ as $t \to T^-$ for every compact operator $K$ on $L^2(\R^3)$. 
Moreover, we have
$$(-\Delta+1)^{-1/4}\gamma(t)(-\Delta+1)^{-1/4}\to (-\Delta+1)^{-1/4}\gamma_*(-\Delta+1)^{-1/4}$$
strongly in the trace class as $t\to T^-$. 

\smallskip

\item[$(ii)$] {\bf Existence and uniqueness of blowup measure:} There exists a unique, positive, regular Borel measure $\mu \in \mathcal{M}(\R^3)$ such that 
$$
\mbox{$\rho_\gamma(t,\cdot) \weakto \mu$ weakly in $\mathcal{M}(\R^3)$ as $t \to T^-$},
$$
and we have that $\mu(\R^3) = \tr\,\gamma_0$.  

\smallskip

\item[$(iii)$] {\bf Minimal mass concentration point:} There is some point $x_*  \in \R^3$ such that 
$$
\mu(\{ x_* \}) \geq \Mcrit(\norm{\gamma_0}),
$$
where $\mu \in \mathcal{M}(\R^3)$ is given in $(ii)$. 
\end{enumerate}

\medskip

Furthermore, if $\gamma_0$ is radially symmetric in the sense that its integral kernel satisfies $\gamma_0(\mathcal{R}x,\mathcal{R}y)=\gamma_0(x,y)$ for every rotation $\mathcal{R}\in SO(3)$, then we have 
$$\mu = M \delta_{x=0} + \tr\,\gamma_*$$ 
for some $M \geq \Mcrit(\norm{\gamma_0})$, and moreover
$$
\mbox{$\1_{|x|\geq R}\,\gamma(t)\, \1_{|x|\geq R}\to \1_{|x|\geq R}\,\gamma_* \,\1_{|x|\geq R}$ strongly in the trace class,}
$$
$$
\mbox{$\1_{|x|\geq R}\,\rho_{\gamma(t)}\to \1_{|x|\geq R}\,\rho_{\gamma_*}$ strongly in $L^1(\R^3)$}
$$
for every $R > 0$, as $t \to T^-$. 
\end{thm}

\medskip

\begin{remark*} \em
If the radially symmetric initial datum $\gamma_0$ additionally satisfies $$\sum_{j=1}^3 \tr \left ( x_j \sqrt{-\Delta+m^2} x_j \gamma_0 \right ) < +\infty,$$ then a convergence statement similar to $(i')$ in Theorem \ref{thm:rad} holds.
\end{remark*}

\bigskip

The rest of the paper is devoted to the proof of Theorems~\ref{thm:nonrad} and~\ref{thm:rad}. The proof of the more general Theorem~\ref{thm:Hartree} is similar and it will not be given here.

\bigskip

\noindent\textbf{Conventions.}
We use $X \lesssim Y$ to denote that $X \leq C Y$, where $C > 0$ is some constant that only depends on the initial datum $u_0 \in H^{1/2}(\R^3)$, on the mass constant $m \geq 0$   in \eqref{eq:pde}, and perhaps on some fixed cutoff function. Since we mostly work in $d=3$ space dimensions, we use the notation $\| u \|_{L^p} \equiv \| u \|_{L^p(\R^3)}$ and $\| u \|_{H^s} \equiv \| u \|_{H^s(\R^3)}$ throughout this paper. For brevity's sake, we will use $u(t)$ or simply $u$ to denote the function $u=u(t,x)$.

\bigskip

\noindent\textbf{Acknowledgments.} E.\,L.~gratefully acknowledges support by a Steno research grant from the Danish science research council (FNU). M.\,L.~gratefully acknowledges support by the European Research Council under the European Community's Seventh Framework Program (FP7/2007--2013 Grant Agreement MNIQS no. 258023) and by the French Ministry of Research (ANR-10-BLAN-0101).

\section{Proof of Theorem \ref{thm:nonrad}}\label{sec:proof_nonradial}

Suppose that $u \in C^0([0,T); H^{1/2}(\R^3))$ solves \eqref{eq:pde} and blows up at finite time $0 < T < +\infty$. Note that we do not impose any symmetry condition on $u=u(t,x)$ throughout this section. 

\subsection*{Step 1 (Existence and Uniqueness of weak $L^2$-limit)}
First, we derive the following a-priori bound 
\begin{equation} \label{ineq:apriori_Hminus1}
\sup_{t \in [0,T)} \| \pt_t u(t) \|_{H^{-1}} \lesssim 1. 
\end{equation}
To show this bound, we note that the first term on the right-hand side in \eqref{eq:pde} satisfies
\begin{equation*}
 \| \sqrt{-\Delta + m^2} \, u(t) \|_{H^{-1}} \lesssim \| u(t) \|_{L^2} \lesssim 1,
\end{equation*} 
which follows from $L^2$-mass conservation. Next, we consider the function
\begin{equation*}
V_{u} := |x|^{-1} \ast |u|^2.
\end{equation*}
For $\phi \in H^1(\R^3)$,  we deduce the bound
\begin{equation*}
\| V_{u(t)} \phi \|_{L^2} \leq \int_{\R^3} \left ( \int_{\R^3} \frac{|\phi(y)|^2}{|x-y|^2} \, dy \right )^{1/2} |u(t,x)|^2 \, dx \leq 4 \| \nabla \phi \|_{L^2} \| u_0 \|_{L^2}^2,
\end{equation*}
where we used Minkowski's inequality together with Hardy's inequality $|x|^{-2} \leq 4 (-\Delta)$ in $\R^3$ and the conservation of $L^2$-mass. Therefore, we obtain the bound
\begin{equation*}
\big  | \big \langle V_{u(t)} u(t), \phi \big \rangle_{L^2} \big | \leq \| V_{u(t)} \phi \|_{L^2} \| u(t) \|_{L^2} \lesssim \| \phi \|_{H^1} \quad \mbox{for $t \in [0,T)$},
\end{equation*}
with any test function $\phi \in H^1(\R^3)$. In summary, we have shown that the right-hand side of \eqref{eq:pde} satisfies 
\begin{equation*}
\left \| \mbox{RHS of \eqref{eq:pde}} \right  \|_{H^{-1}} \lesssim 1, \quad \mbox{for $t \in [0,T)$}, 
\end{equation*}
which implies the desired a-priori bound \eqref{ineq:apriori_Hminus1}. 

Next, by $L^2$-mass conservation, we know that $u\in L^\ii([0,T),L^2(\R^3))$ and we have just shown in \eqref{ineq:apriori_Hminus1} that $\partial_t u\in L^\ii([0,T),H^{-1}(\R^3))$. By a classical interpolation result (see \cite{Lions-58}) this implies that
$$u\in C^0([0,T],H^{-1/2}(\R^3)).$$
In particular, there exists $u_* \in H^{-1/2}(\R^3)$ such that $u(t)\to u_*$ strongly in $H^{-1/2}(\R^3)$ as $t \to T^-$. Since $u(t)$ is also uniformly bounded in $L^2(\R^3)$, we also have that $u(t)\wto u_*$ weakly in $L^2(\R^3)$ as $t \to T^-$. This completes the proof of Part $(i)$ of Theorem \ref{thm:nonrad}.

\subsection*{Step 2 (Existence and Uniqueness of Blowup Measure)}
We begin with the following {\em ``finite speed of propagation estimate'' } for spatially localized parts of the $L^2$-norm of the solution $u=u(t,x)$. We have the following estimate.
\begin{lemma}[Finite speed of propagation] \label{lem:finite}
For $\chi \in W^{1,\infty}(\R^3)$ given, we define the function
$$
M_\chi(t) := \int_{\R^3} |u(t,x)|^2 \chi(x) \, dx, \quad \mbox{for $t \in [0,T)$}.
$$
Then $M_\chi(t)$ is differentiable and satisfies the uniform bound
$$
\left | \frac{d M_\chi(t)}{dt} \right | \lesssim \| \nabla \chi \|_{L^\infty} , \quad \mbox{for $t \in [0,T)$}.
$$
In particular, the function $M_{\chi}(t)$ has a limit as $t \to T^-$.
\end{lemma}

\begin{proof}
First, we assume that $u \in C^0([0,T); H^1(\R^3)) \cap C^1([0,T); L^2(\R^3))$ holds. Then a well-defined calculation shows that $M_\chi(t)$ satisfies 
\begin{equation} \label{eq:heisen}
\frac{d M_\chi(t)}{dt} = i \left \langle u(t), [\sqrt{-\Delta +m^2}, \chi ] u(t) \right \rangle_{L^2}.
\end{equation}
Next, we recall Calder\'on's commutator estimate (see Appendix \ref{app:CC_frac_Sobolev} for more details)
\begin{equation}
\| [\sqrt{-\Delta+m^2}, \chi ] \|_{L^2 \to L^2} \lesssim \| \nabla \chi \|_{L^\infty}.
\label{eq:Stein}
\end{equation} 
Combining this estimate with the conservation of $L^2$-mass, we find that
\begin{equation}
\left | \frac{d M_\chi(t)}{dt} \right | \lesssim \| \nabla \chi \|_{L^\infty} \| u(t) \|_{L^2}^2 \lesssim \| \nabla \chi \|_{L^\infty},
\end{equation}
which is the desired bound, provided that $u(t) \in H^1(\R^3)$ holds. 

By a standard approximation argument and using the local well-posedness of \eqref{eq:pde} in $H^s(\R^3)$ for $s \geq 1/2$ as shown in \cite{Le07}, we conclude that \eqref{eq:heisen} indeed holds for  solutions $u \in C^0([0,T); H^{1/2}(\R^3))$. This completes the proof of Lemma \ref{lem:finite}. \end{proof}

\noindent
Next, we  derive with the following {\em ``tightness''} property.
 
\begin{lemma} \label{lem:tight}
The family of measures $\{ |u(t)|^2 \}_{t \in [0,T)} \subset \mathcal{M}(\R^3)$ is tight. That is, for every $\eps > 0$, there exists some constant $R > 0$ such that
$$
\int_{|x| \geq R} |u(t,x)|^2 \, dx  \leq \eps, \quad \mbox{for $t \in [0,T)$}.
$$ 
\end{lemma}

\begin{proof}
Let $\chi \in W^{1,\infty}(\R^3)$ be a cutoff function with $0 \leq \chi \leq 1$ such that $\chi(x) \equiv 1$ for $|x| \leq 1/2$ and $\chi(x) \equiv 0$ for $|x| \geq 1$. For $R > 0$, we  define $\chi_R(x) := \chi(x/R)$ and we consider  
\begin{equation}
M_R(t) := \int_{\R^3} |u(t,x)|^2 \chi_R  ( x) \, dx.
\end{equation}
By Lemma \ref{lem:finite},  we have the uniform bound 
\begin{equation}
\left | \frac{d M_R(t)}{dt} \right | \lesssim \| \nabla \chi_R \|_{L^\infty} \lesssim \frac{1}{R}.
\end{equation}
Integrating this bound and using the fact that $0 < T < +\infty$, we obtain 
\begin{equation}
M_{R}(t) \geq M_R(0) - \frac{CT}{R}, \quad \mbox{for $t \in [0,T)$},
\end{equation} 
where $C >0$ is some fixed constant. Now, let $\eps > 0$ be given. We choose $R > 0$ sufficiently large such that $M_R(0) \geq \int |u_0|^2 - \eps/2$ and $R \geq 2CT\eps^{-1}$ both hold. By $L^2$-mass conservation and using $\chi_R(x) \leq \1_{\{ |x| \leq  R \} }(x)$, we conclude that
\begin{align*}
\int_{|x| \geq R} |u(t,x)|^2 & = \int_{\R^3} |u_0(x)|^2  - \int_{|x| \leq R} |u(t,x)|^2  \leq \int_{\R^3} |u_0(x)|^2 -M_R(t) \leq \eps,
\end{align*}
for all $t \in [0,T)$. This is the claimed tightness bound. \end{proof}

We now turn to the proof of Part $(ii)$ of Theorem \ref{thm:nonrad}. Let $\{ t_n \}_{n=1}^\infty$ be a sequence of times such that $t_n \to T^-$. By $L^2$-mass conservation and Lemma \ref{lem:tight}, the sequence $\{ |u(t_n)|^2 \}_{n=1}^\infty$ is uniformly bounded $L^1(\R^3)$ and forms a tight family of measures on $\R^3$. By Prokhorov's theorem, we deduce (after passing to a subsequence if necessary) that 
\begin{equation*}
\mbox{$|u(t_n)|^2 \weakto \mu$ weakly in $\mathcal{M}(\R^3)$ as $n \to \infty$,} 
\end{equation*}
where $\mu\in \mathcal{M}(\R^3)$ is some positive Borel measure on $\R^3$. Using the tightness property in Lemma \ref{lem:tight} and $\int_{\R^3} |u(t)|^2 = \int_{\R^3} |u_0|^2$, we readily check that $\mu(\R^3) = \int_{\R^3} | u_0|^2$ holds. In particular, we see that $\mu$ is a regular Borel measure, since $\mu$ is finite and positive.  

To show that $\mu \in \mathcal{M}(\R^3)$ is independent of the sequence $t_n \to T^-$, we use Lemma \ref{lem:finite}  again. Indeed, let $\chi \in W^{1,\infty}(\R)$ be given and consider 
\begin{equation}
M_{\chi}(t) := \int_{\R^3} |u(t,x)|^2 \chi(x) \, dx .
\end{equation} 
By Lemma \ref{lem:finite}, we have $M_\chi(t) \to M^*_\chi$ as $t \to T^-$ with some unique limit $M_\chi^* \in \R$. Hence,
\begin{equation} \label{eq:measure}
M_\chi^* = \int_{\R^3} \chi d \mu = \int_{\R^3} \chi d \widetilde{\mu} ,
\end{equation}
whenever $|u(t_n)|^2 \weakto \mu$ and $|u(\tilde{t}_n)|^2 \weakto \widetilde{\mu}$ weakly in $\mathcal{M}(\R^3)$ as  $t_n \to T^-$ and $\tilde{t}_n \to T^-$.  Since the second equation in \eqref{eq:measure} is valid for all $\chi \in W^{1,\infty}(\R^3)$, we deduce that equality  $\mu = \widetilde{\mu}$ holds for the finite and positive Borel measures $\mu \in \mathcal{M}(\R^3)$ and $\tilde{\mu} \in \mathcal{M}(\R^3)$. This completes the proof of Part $(ii)$ of Theorem \ref{thm:nonrad}.

\subsection*{Step 3 (Existence of Blowup Point)}
We finally establish Part $(iii)$ of Theorem \ref{thm:nonrad}, that is we show existence of (at least) one concentration point $x_* \in \R^3$ such that
\begin{equation} \label{ineq:muatom}
\mu( \{ x_* \} ) \geq \Mcrit.
\end{equation}
We recall that $\Mcrit > 0$ is the critical mass, see Appendix \ref{app:Mcrit}. As a first step, we prove the following weaker result.

\begin{lemma} \label{lem:yt}
There exists a function $y : [0,T) \to \R^3$ such that 
\begin{equation}
\liminf_{t \to T^-} \int_{|x-y(t)| \leq R} |u(t,x)|^2 \, dx \geq \Mcrit,
\end{equation} 
for any $R > 0$. 
\end{lemma}

\begin{remark} \em
 In the context of $L^2$-critical NLS, the analogous result of Lemma \ref{lem:yt} is well-known; see \cite{MeTs90,Na90,We86,HmKe05}. However, the adaptation of the NLS arguments  to the nonlocal pseudo-differential operator $\sqrt{-\Delta+m^2}$ poses some technical difficulties. Therefore we will provide a self-contained proof below.  
 \end{remark}
 

\begin{proof}[Proof of Lemma \ref{lem:yt}]
Let $t_n \to T^-$ be any sequence approaching the blowup time and define the sequence $\bm{v}= \{ v_n \}_{n=1}^\infty \subset H^{1/2}(\R^3)$ by setting 
\begin{equation} \label{def:vn}
v_n(x):= \sigma_n^{-3/2} u(t_n, \sigma_n^{-1} x) , \quad \mbox{where $\sigma_n := \| |\nabla|^{1/2} u(t_n) \|_{L^2}^2$.}
\end{equation}
Notice that $\|  v_n \|_{L^2}=  \|  u_0  \|_{L^2}$ and $\| |\nabla|^{1/2} v_n  \|_{L^2} = 1$ for all $n \in \N$. We then introduce the energy-type functional
\begin{equation}
\mathcal{E}(u) :=  \frac{1}{2} \int_{\R^3} | |\nabla|^{1/2} u|^2  - \frac{1}{4} \Dcou(|u|^2),
\label{eq:def_massless_energy}
\end{equation}
where we denote 
\begin{equation}
\Dcou(|u|^2) := \int_{\R^3} \big ( |x|^{-1} \ast |u|^2 \big ) |u|^2 . 
\end{equation} 
The energy $\cE$ is the same as the original conserved energy $E$ defined in~\eqref{eq:def_energy}, except that the constant $m$ has been set to zero.
Using that $\sigma_n \rightarrow +\infty$ by assumption, an elementary calculation shows that
\begin{equation} \label{E0zero}
\mathcal{E}(v_n) = \sigma_n^{-1} \mathcal{E}(u(t_n)) \rightarrow 0 ,
\end{equation}
since $|\mathcal{E}(u(t_n))| \leq E[u_0] + m \| u_0 \|_{L^2}^2 \lesssim 1$, thanks to conservation of the energy $E$ and $L^2$-mass, as well as the operator estimates $\sqrt{-\Delta+m^2} -m \leq \sqrt{-\Delta} \leq \sqrt{-\Delta + m^2}$. By  \eqref{E0zero} and the normalization constraint $\| |\nabla|^{1/2} v_n \|_{L^2} = 1$ for $n \in \N$, we find that 
\begin{equation} \label{D0zero}
\Dcou(|v_n|^2) \to 2 .
\end{equation}

As a next step, we apply classical compactness arguments (see \cite{Lieb-83,Lions-84a,Lions-84b}) to the bounded sequence $\bm{v} = \{ v_n \}_{n=1}^\infty \subset H^{1/2}(\R^3)$. For the reader's convenience, the adaptation of these arguments to the setting of the fractional Sobolev space $H^{1/2}(\R^3)$ is outlined in Appendix \ref{app:CC_frac_Sobolev}. 
As a first and central ingredient, we consider the ``highest mass of weak limits up to translations and extraction of a subsequence'' of the sequence $\bm{v} = \{v_n\}_{n=1}^\infty \subset H^{1/2}(\R^3)$ which is given by
\begin{equation*}
\mass (\bm{v} ):= \sup\left\{\int_{\R^3}|v|^2\ :\ \exists\{x_k\}\subset\R^3,\ v_{n_k}(\cdot-x_k)\wto v\ \text{weakly in $H^{1/2}(\R^3)$}\right\}.
\label{eq:lambda}
\end{equation*}
Clearly, we have the bounds $0 \leq \mass (\bm{v})) \leq  \int_{\R^3} |u_0|^2$. We now establish the following claim, which will be the heart of the proof of Lemma \ref{lem:yt}.

\begin{claim}
The sequence $\bm{v} = \{ v_n \}_{n=1}^\infty \subset H^{1/2}(\R^3)$ satisfies $\mass (\bm{v} ) \geq \Mcrit$. 
\end{claim}

First, we easily rule out that $\mass ( \bm{v}  )$ vanishes by a simple contradiction argument. Indeed, from Lemma \ref{lem:vanishing} in Appendix~\ref{app:CC_frac_Sobolev}, we know that $\mass( \bm{v})= 0$ if and only if $v_n\to0$ strongly in $L^p(\R^3)$ for all $2<p<3$. By the Hardy--Littlewood--Sobolev inequality, this implies
$$\Dcou(|v_n|^2) \lesssim \|v_n \|^4_{L^{12/5}} \to 0,$$
 which contradicts \eqref{D0zero}. Hence the case $\mass (\bm{v}  )=0$ cannot occur.

To complete the proof of {\bf Claim 1},  it remains to rule out that 
\begin{equation} \label{ineq:lsub}
\mass ( \bm{v} ) < \Mcrit.
\end{equation}
We argue by contradiction and suppose that \eqref{ineq:lsub} is true. To reach the desired contradiction, we make use of a detailed profile decomposition of the bounded sequence $\bm{v} = \{ v_n \}_{n=1}^\infty \subset H^{1/2}(\R^3)$. That is, we write the sequence $v_n$ (after passing to a subsequence if necessary) as a sum of finitely many ``bumps''  receding from each other, plus some error term, which can be made arbitrarily small in an appropriate sense. In the Sobolev space $H^1(\R^d)$, results of this kind are well-known~\cite{Struwe-84,BreCor-85,Lions-87,Gerard-98}. Lemma \ref{lem:splitting} in Appendix~\ref{app:CC_frac_Sobolev} below generalizes these results to the case of the fractional Sobolev space $H^{1/2}(\R^3)$. 

More precisely, by Lemma \ref{lem:splitting}, the following holds:  For every $\eps>0$ fixed, there exist an integer $J \in \mathbb{N}$ and bounded sequences $\bm{v}^1 = \{ v_k^1 \}_{k=1}^\infty,\ldots, \bm{v}^J = \{ v^J_k \}_{k=1}^\infty$ in $H^{1/2}(\R^3)$ and $\bm{r}^J = \{ r^J_k\}_{k=1}^\infty$ in $H^{1/2}(\R^3)$ and sequences $\{ x_k^1\}_{k=1}^\infty,\ldots,\{ x_k^J\}_{k=1}^\infty$ in $\R^3$, as well as a sequence of real positive numbers $R_k\to +\ii$ such that the following holds for some subsequence $\{ v_{n_k} \}$:

\begin{enumerate}
 \item[$(i)$] For every $j=1,...,J$ and $k\geq1$, we have $\mathrm{supp} (v_k^j) \subset B(0,R_k)$ and $v^j_k \weakto v^j$ weakly in $H^{1/2}(\R^3)$ and $v^j_k \to v^j$ strongly in $L^p(\R^3)$ for all $2 \leq p < 3$; 
 
\smallskip

\item[$(ii)$] For every $j=1,...,J$ and $k\geq1$, we have $\mathrm{supp} ( r_k^{J}) \subset \R^3\setminus\cup_{j=1}^JB(x_k^j,2R_k)$ and  
$$\mass (\bm{r}^J)\leq \eps;$$

\smallskip

\item[$(iii)$] $\displaystyle\lim_{k\to\ii}\bigg\|v_{n_k}-\sum_{j=1}^Jv^j_k(\cdot-x_k^J)-r_k^J\bigg\|_{H^{1/2}(\R^3)}=0$;

\smallskip

\item[$(iv)$] $\displaystyle\lim_{k\to\ii}\bigg(\pscal{v_{n_k},\sqrt{-\Delta}\,v_{n_k}}-\sum_{j=1}^J\pscal{v^j_{k},\sqrt{-\Delta}\,v^j_{k}}-\pscal{r^J_{k},\sqrt{-\Delta}\,r_{k}^J}\bigg)=0$;

\smallskip

\item[$(v)$] $\displaystyle|x_k^j-x_k^{j'}|\geq 5R_k$ for all $k\geq1$ and $j\neq j'$; 

\smallskip

\item[$(vi)$] $\displaystyle \sup_{j=1,...,J} \left ( \int_{\R^d}|v^j|^2 \right ) \geq \mass (\bm{v})-\eps$.
\end{enumerate}

Up to the correct choice of $J \in \N$, the number $\eps >0$ can be made a small as desired, without changing the functions $v^j$.
Moreover, we can assume in the above decomposition that $v^j\neq0$ for all $j=1,...,J$, since otherwise the corresponding $v_k^j$ can be put into the error term $r_k^J$. Since $\mass (\bm{v}) > 0$ by assumption and choosing $\eps > 0$ sufficiently small, we see from $(vi)$ that there must be at least one $v^j$; or, saying differently, we see that $J\geq1$ holds. Also notice that, by the support properties stated in $(i)$ and $(ii)$, we have $v_{n_k}(\cdot+x_k^j)\wto v^j$ weakly in $H^{1/2}(\R^3)$ for every $j=1,...,J$. Therefore we deduce
$$ \int_{\R^3}|v^j|^2\leq \mass ( \bm{v})<\Mcrit, \quad \mbox{for all $j =1,\ldots, J,$}$$
provided that our assumption \eqref{ineq:lsub} holds.

Now we apply the decomposition result stated above to our problem. First, we claim that
\begin{equation}
\cE(v_{n_k})=\sum_{j=1}^J \cE(v_{k}^j)+\cE(r_{k}^J)+o(1)_{k\to\ii},
\label{eq:splitting_energy} 
\end{equation}
where we recall that $\cE=\cE(\cdot)$ is the translation-invariant massless energy defined in \eqref{eq:def_massless_energy}. The splitting of the kinetic energy is exactly stated in $(iv)$, whereas the potential energy term can be estimated as
\begin{align*}
|\Dcou(|f|^2)-\Dcou(|g|^2)|&=|\Dcou(|f|^2-|g|^2,|f|^2+|g|^2)|\\
&\lesssim \norm{f-g}_{L^{12/5}}\norm{f+g}_{L^{12/5}}\norm{|f|^2+|g|^2}_{L^{6/5}},
\end{align*}
using the Hardy--Littlewood--Sobolev inequality. From $(iii)$ we thus deduce that
$$\Dcou(|v_{n_k}|^2)=\Dcou \left(\bigg|\sum_{j=1}^Jv_k^j(\cdot-x^j_k)+r^J_k\bigg|^2\right)+o(1)_{k\to\ii}.$$
Finally, using the support properties of the $v^j_k$'s and of $r_k^J$, we easily see that 
$$\Dcou \left(\bigg|\sum_{j=1}^Jv_k^j(\cdot-x^j_k)+r^J_k\bigg|^2\right)=\sum_{j=1}^J \Dcou(|v_k^j|^2)+\Dcou(|r^J_k|^2)+o(1)_{k\to\ii},$$
whence \eqref{eq:splitting_energy} follows.

By the strong convergence of $v_k^j$ in $L^{12/5}(\R^3)$ as $k\to\ii$ for every $j=1,...,J$, we have in fact that $\Dcou(|v_k^j|^2)\to \Dcou(|v^j|^2)$ holds. Combining this with Fatou's lemma for the kinetic energy, we deduce that
\begin{equation}
 \liminf_{k\to\ii}\cE(v_k^j)\geq \cE(v^j), \quad \forall j = 1, \ldots, J.
\label{eq:wlsc_1rst_piece} 
\end{equation}
Hence we obtain
\begin{equation}
0=\lim_{k\to\ii}\cE(v_{n_k})\geq \sum_{j=1}^J \cE(v^j)+\liminf_{k\to\ii}\cE(r_{k}^J).
\label{eq:lower_bound_splitting_energy} 
\end{equation}

From Appendix \ref{app:Mcrit}, we recall that the critical mass $\Mcrit >0$ provides the best constant in the inequality 
$$\frac{\Dcou(|f|^2)}4\leq \frac{\int_{\R^3}|f|^2}{2\Mcrit}\int_{\R^3}||\nabla|^{1/2}f|^2 ,$$
which implies, for all $j =1,\ldots,J$, the lower bound
$$\cE(v^j)\geq \frac12 \left(1-\frac{\int_{\R^3}|v^j|^2}{\Mcrit}\right)\int_{\R^3}||\nabla|^{1/2}v^j|^2>0,$$
since $v^j\neq0$ and $\int_{\R^3}|v^j|^2<\Mcrit$. Recalling now \eqref{eq:splitting_energy}, we obtain that
\begin{equation}
\liminf_{k\to\ii}\cE(r_k^J)\leq -\sum_{j=1}^J\cE(v^j) \leq -\cE(v^1)<0.
\label{eq:upper_bound_energy} 
\end{equation}
Next, from Lemma~\ref{lem:subcritical_bound} in Appendix \ref{app:CC_frac_Sobolev}, we invoke the estimate
$$\limsup_{n\to\ii}\| f_n \|_{L^{8/3}}^{8/3}\lesssim  \left(\limsup_{n\to\ii} \| f_n \|_{H^{1/2}}^2\right) \, \mass (\bm{f} )^{1/3}$$
for any fixed bounded sequence in $\bm{f} = \{ f_n \}_{n=1}^\infty \subset H^{1/2}(\R^3)$. On the other hand, the H\"older and Hardy--Littlewood--Sobolev inequalities imply that 
\begin{equation}
\Dcou(|r_k^J|^2)\lesssim \norm{r_k^J}^4_{L^{12/5}}\lesssim \,\norm{r_k^J}^{4/3}_{L^{2}}\norm{r_k^J}^{8/3}_{L^{8/3}}.
\end{equation}
Hence we deduce that
$$\liminf_{k\to\ii}\cE(r_k^J)\geq -C \left(\limsup_{k\to\ii}\norm{r_k^J}^{4/3}_{L^{2}}\right)\left(\limsup_{k\to\ii}\norm{r_k^J}^{2}_{H^{1/2}}\right)\,\mass\big(\bm{r}^J)^{1/3}.$$
Note that, by $(iii)$ and $(iv)$ in the profile decomposition of $\bm{v}= \{ v_n \}_{n=1}^\infty$, we find  
$$\limsup_{k\to\ii}\norm{r_k^J}_{L^{2}}\leq \limsup_{k\to\ii}\norm{v_{n_k}}_{L^{2}} \lesssim \norm{u_0}_{L^{2}} \lesssim 1 $$
and
$$\limsup_{k\to\ii}\norm{r_k^J}^{2}_{H^{1/2}}\leq \limsup_{k\to\ii}\norm{v_{n_k}}^{2}_{H^{1/2}} \lesssim 1+\norm{u_0}_{L^{2}}^2 \lesssim 1.$$
Using $(ii)$, we thus arrive at the following estimate
\begin{equation}
\liminf_{k\to\ii}\cE(r_k^J)\geq -C \eps^{1/3},
\label{eq:final_bound_energy}
\end{equation}
where the constant $C > 0$ only depends on $\| u_0 \|_{L^2}$. By increasing $J\in \N$ if necessary (which does not change the limiting functions $v^j$), we can assume that $\eps >0$ is as small as desired. But then \eqref{eq:final_bound_energy} contradicts \eqref{eq:upper_bound_energy}. Thus assumption \eqref{ineq:lsub} cannot hold; and this completes the proof of {\bf Claim 1}.

\medskip

Now we are ready to finish the proof of Lemma \ref{lem:yt} by the following standard arguments, which also apply in the context of $L^2$-critical NLS; see, e.\,g., \cite{MeTs90,Na90,We86,HmKe05}. Recall that $\bm{v} = \{ v_n \}_{n=1}^\infty$ was defined in \eqref{def:vn} and that $t_n \to T^-$. Note that $\mass(\bm{v}) \geq \Mcrit$ by {\bf Claim 1}. Hence we can find, for any $\eps > 0$, a sequence $\{ x_k \}_{k=1}^\infty$ in $\R^3$ such that (after passing to a subsequence if necessary) we have  $v_{n_k}(\cdot-x^1_k)\wto v$ weakly in $H^{1/2}(\R^3)$ and
\begin{equation} \label{eq:vmass}
\int_{\R^3} |v|^2 \geq \Mcrit-\eps .
\end{equation}
Let $A > 0$ and $\eps > 0$ now be fixed. Since $v_{n_k}(\cdot-x_k)\to v$ strongly in $L^{2}_{\rm loc}(\R^3)$ by Rellich's local compactness and undoing the rescaling in \eqref{def:vn}, we deduce that
\begin{equation}
\lim_{k \to \infty} \int_{|x-y_k|\leq \sigma_{n_k}^{-1} A} |u(t_{n_k}, x)|^2 = \int_{|x| \leq A} |v|^2  ,
\end{equation}
for any $A > 0$ fixed and the sequence of translations $y_k = -\sigma_{n_k} x_k$. Next, let $\lambda=\lambda(t)$ be any function such that 
$$
\mbox{$\lambda(t) \to 0$ and $\displaystyle\lambda(t) \int_{\R^3} | |\nabla|^{1/2} u(t) |^2 \to +\infty$ as $t \to T^-$,}
$$
e.\,g., we can choose $\lambda(t) =\left(\| |\nabla|^{1/2} u(t) \|_{L^2}^2\right)^{-1/2}$. In particular, we have $\lambda(t_n) \geq \sigma_n^{-1} A$ for  $n \geq 1$ sufficiently large. By eliminating the sequence $\{ y_n \}_{n=1}^\infty$ through taking the supremum over $y \in \R^3$, we conclude that
\begin{equation*}
\lim_{k\to\ii} \sup_{y \in \R^3} \int_{|x-y| \leq \lambda(t_{n_k})} |u(t_{n_k},x)|^2   \geq \int_{|x| \leq A} |v|^2 .
\end{equation*}
Since $A > 0$ and $t_n\to T^-$ can be chosen arbitrarily, we obtain from \eqref{eq:vmass} and, by sending $\eps \to 0$, the lower bound
\begin{equation}
\liminf_{t\to T^-} \sup_{y \in \R^3} \int_{|x-y| \leq \lambda(t)} |u(t,x)|^2   \geq \Mcrit. 
\end{equation}
Notice that, for any $t \in [0,T)$ fixed, the function $y \mapsto \sup_{y \in \R^3} \int_{|x-y|\leq \lambda(t)} |u(t,x)|^2$ is continuous and vanishes as $|y| \to +\infty$. Hence there exists $y(t) \in \R^3$ such that
\begin{equation}
\liminf_{t \to T^-} \int_{|x-y(t)| \leq \lambda(t)} |u(t,x)|^2   \geq \Mcrit.
\end{equation}
Furthermore, for any $R > 0$ given, the inclusion $B(x-y(t), \lambda(t)) \subseteq B(x-y(t),R)$ holds for $|T-t|>0$ sufficiently small, since $\lambda(t) \to 0$. The proof of Lemma \ref{lem:yt} is now complete. \end{proof}

Now, we use Lemma \ref{lem:tight} to derive the following \emph{a priori} bound for the function $y : [0,T) \to \R^3$, which was provided by Lemma \ref{lem:yt}. 

\begin{lemma} \label{lem:fini}
Let $y : [0,T) \to \R^3$ satisfy the conclusion of Lemma \ref{lem:yt}. Then we have
$$
\sup_{t \in [t_0, T)} |y(t)| \lesssim 1 ,
$$
for some $t_0 \in [0,T)$ sufficiently close to $T$.
\end{lemma}

\begin{proof}
By Lemma \ref{lem:tight}, there exists $R_* > 0$ sufficiently large that 
\begin{equation} \label{ineq:tight}
\int_{|x| \geq R_*} |u(t,x)|^2 \leq \frac{1}{2} \Mcrit  \quad \mbox{for all $t \in [0,T)$}.
\end{equation}
Suppose now that $|y(t_n)| \to +\infty$ along some sequence $t_n \to T^-$. Without loss of generality, we can assume that $|y(t_n)| \geq 2R_*$ for all $n \in \N$, which implies that $\{ | x-y(t_n) | \leq R_* \} \subseteq \{ |x| \geq R_* \}$. Hence, by \eqref{ineq:tight}, we obtain
\begin{equation*} 
 \int_{|x-y(t_n)| \leq R_*} |u(t_n, x)|^2  \leq \int_{|x| \geq R_*} |u(t_n,x)|^2  \leq \frac{1}{2} \Mcrit.
\end{equation*} 
However, by assumption on $y(t)$, we also find
$$
\liminf_{n \to \infty} \int_{|x-y(t_n)| \leq R_*} |u(t_n,x)|^2 \geq \Mcrit,
$$
which is the desired contradiction. \end{proof}

Finally, we conclude the proof of Part $(iii)$ of Theorem \ref{thm:nonrad} as follows. Let $y : [0,T) \to \R^3$ be given by Lemma \ref{lem:yt}. Thanks to Lemma \ref{lem:fini}, we can assume that $y(t_n) \to x_*$ for some $x_* \in \R^3$ and some sequence $t_n \to T^-$. Let $R > 0$ be given. Since $|u(t_n)|^2 \weakto \mu$ weakly in $\mathcal{M}(\R^3)$, we deduce that
\begin{align*}
 \mu( \{ |x - x_*| \leq R \} ) & \geq \limsup_{n \to \infty} \int_{|x-x_*| \leq R} |u(t_n,x)|^2 \geq \\ & \geq \liminf_{n \to \ii} \int_{|x-x_*| \leq R} |u(t_n,x)|^2 \geq \Mcrit .
\end{align*}
Here the last inequality readily follows from Lemma \ref{lem:yt} and the fact that $y(t_n) \to x_*$.  Since $R > 0$ was arbitrary, this shows that $\mu(\{ x_*\}) \geq \Mcrit$ holds for the finite measure $\mu \in \mathcal{M}(\R^3)$. The proof of Theorem \ref{thm:nonrad} is  now complete. \hfill $\blacksquare$

\section{Proof of Theorem \ref{thm:rad}}\label{sec:proof_radial}

Suppose that the initial datum $u_0 = u_0(|x|) \in H^{1/2}(\R^3)$ is radially symmetric and that its corresponding radial solution $u \in C^0([0,T); H^{1/2}(\R^3))$ of \eqref{eq:pde} blows up at finite time $0 < T < +\infty$. Note that, by uniqueness of the solution, we deduce that $u=u(t,|x|)$ is radially symmetric for all $t \in [0,T)$. 

Let $\zeta \in C^\infty(\R^3)$ be a radial and smooth cutoff-function satisfying $0 \leq \zeta \leq 1$ and
\begin{equation}
\zeta(|x|) \equiv 0 \quad \mbox{for $|x| \leq 1$}, \quad \zeta(|x|) \equiv 1 \quad \mbox{for $|x| \geq 2$}. 
\end{equation}
For $R> 0$ given, we define the rescaled version $\zeta_R := \zeta(\cdot/R)$ and we consider the function $u_R : [0,T) \times \R^3 \to \C$ given by
\begin{equation}
u_R(t,x) := \zeta_R(x) u(t,x) .
\end{equation}
A calculation shows that $u_R \in C^0([0,T); H^{1/2}(\R^3))$ satisfies 
\begin{equation} \label{eq:ur}
i \partial_t u_R = \sqrt{-\Delta +m^2} \, u_R + V_u u_R + F_R ,
\end{equation}
where 
\begin{equation}
V_u := - |x|^{-1} \ast |u|^2, \quad F_R:= [\zeta_R,\sqrt{-\Delta+m^2}] u .
\end{equation}
Recall that $[A,B] \equiv AB-BA$ denotes the commutator of $A$ and $B$.  By Duhamel's principle, we deduce from \eqref{eq:ur} that
\begin{equation} \label{eq:Duhamel}
u_R(t) = U(t) u_R(0) - i \int_0^t U(t-s) \bigg(V_u(s) u_R(s) + F_R(s)\bigg) \, ds,
\end{equation}
where $U(t) = e^{-it \sqrt{-\Delta +m^2}}$ denotes the free propagator generated by $\sqrt{-\Delta+m^2}$. Next, by $L^2$-mass conservation for $u=u(t,x)$ and using the commutator estimate $[\zeta_R, \sqrt{-\Delta+m^2} ] \leq C \| \nabla \zeta_R \|_\infty$, we deduce 
\begin{equation} \label{ineq:F}
\| F_R \|_{L^\infty_t L^2_x} \lesssim \| \nabla \zeta_R \|_{L^\infty}  \| u_0 \|_{L^2} \lesssim \frac{1}{R} .
\end{equation}
As for the other term $V_u(s) u_R(s)$, we first use Newton's theorem to derive the following pointwise estimate on $V_u$,
\begin{equation}
V_u(t,x)=\int_{\R^3}\frac{|u(t,y)|^2}{|x-y|}dy
=\int_{\R^3}\frac{|u(t,y)|^2}{\max(|x|,|y|)}dy
\leq \frac{\displaystyle\int_{\R^3}|u(t,y)|^2\,dy}{|x|}
=\frac{\displaystyle\int_{\R^3}|u_0|^2}{|x|},
\label{eq:Newton}
\end{equation}
by radiality of $u(t)$ and mass conservation. Using this bound and the fact that $V_u u_R = V_u \zeta_{R/2} u_R$ holds, thanks to the support properties of $\zeta_R$, we deduce that
\begin{equation} \label{ineq:VR}
\| V_u u_R\|_{L^\infty_t L^2_x} \leq \| V_u \zeta_{R/2} \|_{L^\infty_t L^2_x} \| u_R \|_{L^\infty_t L^2_x} \leq \left ( 2\,\frac{\int_{\R^3} |u_0|^2}{R} \right ) \int_{\R^3} | u_0|^2 \lesssim \frac{1}{R}.
\end{equation}
Combining now \eqref{ineq:F} and \eqref{ineq:VR}, we can pass to the limit $t \to T^-$ in \eqref{eq:Duhamel} and conclude that $u_R(t)$ converges strongly in $L^2(\R^3)$ as $t \to T^-$. From this fact and that $u(t) \weakto u_*$ weakly in $L^2(\R^3)$ by Theorem \ref{thm:nonrad}, we obtain that
\begin{equation} \label{cv:L2}
\mbox{$u(t) \to u_*$ strongly in $L^2( \{ |x| \geq R \})$ as $t \to T^-$},
\end{equation}
for any $R >0$ fixed. This completes the proof of Part $(i)$ in Theorem \ref{thm:rad}.

To prove Part $(ii)$ of Theorem \ref{thm:rad}, we deduce from the general result in Theorem \ref{thm:nonrad} and the strong convergence of $u(t)$ in $L^2( \{ |x| \geq R \})$ that the limiting measure satisfies
\begin{equation}
\mu = M \delta_{x=0} + |u_*|^2
\end{equation}
with $M= \mu(\{ 0 \}) \geq \Mcrit$. 
This completes the proof of Part $(ii)$ of Theorem \ref{thm:rad}.

Under the additional assumption that $x \, u_0 \in H^{1/2}(\R^3)$ holds, we can use the virial estimate in \cite[Eqn.~(1.9)]{FrLe07} to conclude that
\begin{equation}
\sum_{j=1}^3 \left \langle u(t), x_j \sqrt{-\Delta+m^2} x_j u(t) \right \rangle \leq 2E[u_0] t^2 + C_1 t + C_2 \lesssim 1 , \quad \mbox{for $t \in [0,T)$}.
\end{equation}
Using this estimate, we conclude that $x u(t,\cdot) \weakto x u_*$ weakly in $H^{1/2}(\R^)$ as $t \to T^-$. From this, it is easy to establish that $(i')$ holds. The proof of Theorem \ref{thm:rad} is now complete.~\hfill $\blacksquare$

\appendix 
\section{Definition of the Critical Mass}
\label{app:Mcrit}

From \cite{LiYa87,Le07} we recall the interpolation estimate
\begin{equation} \label{ineq:GN}
\int_{\R^3} \big ( |x|^{-1} \ast |f|^2 \big ) |f|^2 \leq  C_{\mathrm{opt}} \left ( \int_{\R^3} | |\nabla|^{1/2} f |^2 \right ) \left ( \int_{\R^3} |f|^2 \right ) ,
\end{equation}
for all $f \in H^{1/2}(\R^3)$, where $C_{\mathrm{opt}} > 0$ denotes the optimal constant. The critical mass $\Mcrit > 0$ is now defined by
\begin{equation}
\Mcrit := \frac{2}{C_{\mathrm{opt}}} = \inf_{f \in H^{1/2}, f \not \equiv 0} \frac{2 \left ( \int_{\R^3} | |\nabla|^{1/2} f |^2 \right ) \left ( \int_{\R^3} |f|^2 \right )}{\int_{\R^3} \big ( |x|^{-1} \ast |f|^2 \big ) |f|^2}   .
\end{equation}
As a direct consequence, we obtain the following bound for the energy
\begin{equation}
E[u] \geq \frac{1}{2} \left ( 1- \frac{ \int_{\R^3} |u|^2 }{\Mcrit} \right ) \int_{\R^3} | |\nabla|^{1/2} u|^2 .
\end{equation}
In conjunction with the local well-posedness and the conservation of $E[u]$ and $M[u] = \int |u|^2$ for the evolution equation \eqref{eq:pde}, this bound readily yields the global well-posedness criterion stated in \eqref{eq:glo}. Furthermore, it is not difficult to see that $\Mcrit > 0$ satisfies \eqref{eq:energy_critical_mass}.

Following arguments in \cite{LiYa87, Le07}, we can deduce that there exists a positive, radially symmetric minimizer $Q= Q(|x|) > 0$ in $H^{1/2}(\R^3)$ for inequality \eqref{ineq:GN}. Any such optimizer $Q$ is also referred to as a {\em ground state} for inequality \eqref{ineq:GN}. After a suitable rescaling $Q \mapsto a Q(b \cdot)$ with some $a > 0$ and $b > 0$, any such ground state $Q \in H^{1/2}(\R^3)$ is found to satisfy
\begin{equation} \label{eq:Qgr}
\sqrt{-\Delta} \, Q + Q - (|x|^{-1} * |Q|^2 ) Q = 0.
\end{equation}
We have the following fact.

\begin{lemma}\label{lem:highest_mass}
If $Q \in H^{1/2}(\R^3)$ with $Q \not \equiv 0$ satisfies equation \eqref{eq:Qgr} and optimizes inequality \eqref{ineq:GN}, then  $\int_{\R^3} |Q|^2 = \Mcrit$ holds. 
\end{lemma}

\begin{proof}
A simple bootstrap argument using \eqref{eq:Qgr} shows that $Q \in H^s(\R^3)$ for all $s \geq 1/2$. Moreover, using the pointwise decay estimate for the resolvent $(\sqrt{-\Delta} + 1) \leq \langle x \rangle^{-4}$ in $\R^3$, we can deduce the bounds $|Q(x)|, |\nabla Q(x)| \lesssim \langle x \rangle^{-4}$. Hence we can multiply \eqref{eq:Qgr} with $\overline{Q}+x \cdot \nabla \overline{Q}$ and integrate by parts. Doing so and after taking the real part, we obtain the ``Pohozaev'' identity
$$
2 \int_{\R^3} ||\nabla|^{1/2} Q|^2 -  \int_{\R^3} \big (|x|^{-1} \ast |Q|^2 \big ) |Q|^2 = 0.
$$
Since $Q \in H^{1/2}(\R^3)$ also optimizes \eqref{ineq:GN}, we deduce  that $\int_{\R^3} |Q|^2 = 2/C_{\mathrm{opt}}=\Mcrit$ holds. \end{proof}

\section{Localization and Lack of Compactness in $H^s(\R^d)$}
\label{app:CC_frac_Sobolev}

The study of locally compact sequences (e.\,g., bounded sequences in $H^1(\R^d)$) which lack global compactness on $\R^d$ (due to translations) is a very classical subject. In this appendix, we classify such loss of global compactness on $\R^d$ in the setting of fractional Sobolev spaces $H^s(\R^d)$ for $0<s\leq1$. Most of the following results are well known to experts. Here our presentation mainly follows ideas of Lieb~\cite{Lieb-83} and Lions~\cite{Lions-84a,Lions-84b} developed for $H^1(\R^d)$, which we combine with localization formulas for fractional Laplacians that we previously introduced in~\cite{LeLe10}. Of course, the results presented below can be  readily  generalized to $H^s(\R^d)$ for any $s> 0$, by a simple induction argument on $[s]$.

In the following, let $d\geq1$ be a fixed space dimension.

\subsection{Commutator bounds and localization formulas}
We start with some commutator estimates for the pseudo-differential operators $(1-\Delta)^{\frac{s}{2}}$ and $(-\Delta)^\frac{s}{2}$. In fact, there are multiple useful bounds in the literature for commutators taking the form $[(a-\Delta)^{\frac{s}{2}},\chi]$, where $\chi$ a sufficiently smooth function on $\R^d$. Here, we are interested in the following commutator bound for $0<s \leq 1$.

\begin{lemma}[Commutator bounds]\label{lem:commutator_bound}
Let $0<s\leq 1$. There exists a constant $C >0$ such that, for all $\chi \in \dot{W}^{1,\infty}(\R^d)$, we have the commutator estimate
\begin{equation}
\norm{\left[(1-\Delta)^{\tfrac{s}2}\,,\, \chi\right]}\leq C\norm{\nabla\chi}_{L^\ii(\R^d)},
\label{eq:commutator_bound}
\end{equation}
where the norm on the left side is the usual operator norm on $L^2(\R^d)$.
\end{lemma}

\begin{remark} \em
For $s=1$, the estimate~\eqref{eq:commutator_bound} is due to Calder\'on~\cite{Calderon-65} by singular integral operator (SIO) and complex variable methods; see also \cite{St93} for a modern exposition of this subject matter. Below we provide a simple proof when $0<s<1$ without resorting to SIO methods. 
\end{remark}

\begin{proof}[Proof of Lemma~\ref{lem:commutator_bound}]
We follow some arguments in \cite{LeLe10}. First, we recall the well-known formula
\begin{equation}
x^s=\frac{\sin(\pi s)}\pi\int_0^\ii \frac{x}{x+t}\;t^{s-1}\,dt .
\label{eq:integral_formula} 
\end{equation}
By this formula and the spectral theorem applied to $p=-i\nabla$, we obtain
\begin{align*}
\left[(1-\Delta)^{\tfrac{s}2}\,,\, \chi\right]&=\frac{\sin(\pi s/2)}\pi\int_0^\ii \frac{1}{1+|p|^2+t}\big[|p|^2,\chi\big]\frac{1}{1+|p|^2+t}\, t^{\frac{s}2}\,dt\\
&=-i\;\frac{\sin(\pi s/2)}\pi\int_0^\ii \!\!\frac{1}{1+|p|^2+t}\big(p\cdot(\nabla\chi)+(\nabla\chi)\cdot p\big)\frac{1}{1+|p|^2+t}\, t^{\frac{s}2}\,dt.
\end{align*}
Next, we observe
\begin{equation*}
\norm{\frac{p}{1+|p|^2+t}\cdot(\nabla\chi)\frac{1}{1+|p|^2+t}}\leq \norm{\frac{|p|}{1+|p|^2+t}}\norm{\nabla\chi}\norm{\frac{1}{1+|p|^2+t}} =\frac{\norm{\nabla\chi}_{L^\ii(\R^d)}}{2(1+t)^{\frac32}},
\end{equation*}
which gives the estimate
$$\norm{\left[(1-\Delta)^{\tfrac{s}2}\,,\, \chi\right]}\leq \frac{\sin(\pi s/2)}\pi \norm{\nabla\chi}_{L^\ii(\R^d)} \int_0^\ii\frac{t^{\frac{s}2}}{(1+t)^{\frac32}}\,dt .$$
This completes the proof of~\eqref{eq:commutator_bound} for $0<s<1$. As we remarked above, the desired bound for $s=1$ can be found in \cite{Calderon-65,St93}.
\end{proof}

\begin{remark} \label{rem:comm} \em
Writing $\nabla\chi=|p|^{(1-s)/2-\eps}|p|^{(s-1)/2+\eps}\nabla\chi$ and using the Hardy--Littlewood--Sobolev inequality, our proof can obviously be generalized to show that 
\begin{equation}
\norm{\left[(1-\Delta)^{\frac{s}2}\,,\, \chi\right]}\leq C\norm{\nabla\chi}_{L^p(\R^d)}
\end{equation}
for all $d/(1-s)<p\leq\ii$.

When $(1-\Delta)^{s/2}$ is replaced with $(-\Delta)^{s/2}$, it is known~\cite{Calderon-65,CoiMey-75} for $s=1$ that 
\begin{equation}
\norm{\left[(-\Delta)^{\frac{1}2}\,,\, \chi\right]}\leq C\norm{\nabla\chi}_{L^\ii(\R^d)}.
\end{equation}
For $0<s<1$, our proof of Lemma~\ref{lem:commutator_bound} above easily shows that $\big[(-\Delta)^{\frac{s}{2}}\,,\, \chi\big]$ is a bounded operator when, for instance, $\nabla\chi\in L^{d/(1-s)-\eps}(\R^d)\cap L^{d/(1-s)+\eps}(\R^d)$.
\end{remark}

Next, we discuss bounds for double commutators of the form $[\chi,[\chi,(1-\Delta)^{s}]]$, which naturally appear when localizing a fractional kinetic energy. The following \emph{localization formula} is especially useful when dealing with an infinite partition of unity. Its proof is again a simple calculation based on the integral representation~\eqref{eq:integral_formula}, see~\cite{LeLe10}.

\begin{lemma}[Localization in fractional Sobolev spaces~\cite{LeLe10}]\label{lem:localization}
Let $\chi\in W^{1,\ii}(\R^d)$ be a real-valued function and suppose $0<s<1$. 
Then we have the localization formula
\begin{align}
\frac{[\chi\,,\,[\chi\,,\,(1-\Delta)^{s}]]}2&=\frac{\chi^2(1-\Delta)^{s}+ (1-\Delta)^{s}\chi^2}2-\chi (1-\Delta)^{s}\chi\nonumber\\
&=-\frac{\sin(\pi s)}\pi\;\int_0^\ii \frac{1}{t+1-\Delta}|\nabla\chi|^2\frac{1}{t+1-\Delta}t^s\,dt+L_\chi,
\label{eq:localization}
\end{align}
where 
$$L_\chi=\frac{\sin(\pi s)}\pi\;\int_0^\ii \frac{1}{t+1-\Delta}[-\Delta,\chi]\frac{1}{t+1-\Delta}[\chi,-\Delta]\frac{1}{t+1-\Delta}\;t^s\,dt$$ 
is a nonnegative operator satisfying 
\begin{equation}
0\leq L_\chi\leq 4s\norm{\nabla\chi}^2_{L^\ii(\R^d)}.
\label{eq:bound_L_chi} 
\end{equation}
\end{lemma}

If we are given a smooth partition of unity $(\chi_k)_{k \in \Z}$ such that $\sum_{k \in \Z} (\chi_k)^2\equiv1$ on $\R^d$, we obtain from~\eqref{eq:localization} that
\begin{align}
(1-\Delta)^{s}-\sum_k\chi_k(1-\Delta)^{s}\chi_k&\geq -\frac{\sin(\pi s)}\pi\int_0^\ii \frac{1}{t+1-\Delta}\sum_k|\nabla\chi_k|^2\frac{1}{t+1-\Delta}t^s\,dt\nonumber\\
&\geq -s \norm{\sum_k|\nabla\chi_k|^2}_{L^\ii(\R^d)}.\label{eq:frac_IMS}
\end{align}
In the second line, we have used that, for any $f \in L^\infty(\R^d)$, 
\begin{equation*}
\norm{\int_0^\ii \frac{1}{t+1-\Delta}f\frac{1}{t+1-\Delta}t^s\,dt}\leq \norm{f}_{L^\ii(\R^d)}\int_0^\ii\frac{t^s\,dt}{(1+t)^2}=\frac{\pi s}{\sin(\pi s)}\norm{f}_{L^\ii(\R^d)}.
\end{equation*}
The bound \eqref{eq:bound_L_chi} on $L_\chi$ is obtained in a similar manner, using that $[-\Delta,\chi]=-\nabla\cdot(\nabla\chi)-(\nabla\chi)\cdot\nabla$ holds.
Note that inequality \eqref{eq:frac_IMS} generalizes the well-known {\em IMS localization formula} (see \cite{CycFroKirSim-87}) given by
$$(-\Delta)-\sum_k\chi_k(-\Delta)\chi_k=-\sum_k|\nabla\chi_k|^2 . $$
Note that the bound \eqref{eq:frac_IMS} is indeed optimal in the limit $s\to1$.

\begin{remark} \em
Lemma~\ref{lem:localization} implies the following double commutator bound
\begin{equation}
\norm{[\chi,[\chi,(1-\Delta)^{s}]]}\leq 8s\norm{\nabla\chi}_{L^\ii(\R^d)}^2.
\label{eq:double_commutator_bound}
\end{equation}
Estimates of this form have been proved first by Coifman and Meyer in \cite{CoiMey-75}. In the zero mass case, one can easily show using our method that $[\chi,[\chi,(-\Delta)^{s}]]$ is bounded when, for instance, $\nabla\chi\in L^{d/(1-s)-\eps}(\R^d)\cap L^{d/(1-s)+\eps}(\R^d)$. 
\end{remark}

\subsection{Lack of compactness for bounded sequences in $H^s(\R^d)$}
Let $0 < s \leq 1$ be given. Suppose that $\bm{u} = \{u_n\}_{n=1}^\infty$ bounded sequence in $H^s(\R^d)$. We define
\begin{equation}
\mass \big(\bm{u}) := \sup\left\{  \int_{\R^d} |u|^2  :  \exists\{x_k\}\subset\R^3,\ u_{n_k}(\cdot-x_k)\wto u\ \text{weakly in $H^{s}(\R^d)$}\right\} ,
\end{equation}
which (intuitively speaking) corresponds to the \emph{``highest $L^2$-mass of weak limits of $\bm{u}$ up to translations and extraction of a subsequence.''} We remark that the definition of $\mass(\bm{u})$ is inspired by a seminal paper of Lieb~\cite{Lieb-83}.  Note that, by means of Rellich's compactness theorem, it can be verified that 
\begin{equation}
\mass ( \bm{u} )=\lim_{R\to\ii}\limsup_{n\to\ii} \left ( \sup_{y\in \R^d}\int_{|x-y| \leq R}|u_n(x)|^2 \right ),
\label{eq:formula_m_Lions}
\end{equation}
a formula that is closer to ideas of Lions~\cite{Lions-84a,Lions-84b}, which employs the sequence of L\'evy concentration functions $\{Q_nÊ\}_{n=1}^\infty$ given by 
$$
Q_n(R) := \sup_{y \in \R^d} \int_{|x-y|\leq R } |u_n(x)|^2 , \quad \mbox{for $R \in [0,\infty)$}.
$$
The purpose of $\mass (\bm{u} )$ is to detect the largest ``bump'' of $L^2$-mass in the sequence $\bm{u}=\{u_n\}_{n=1}^\infty$. Note that it is possible that such largest ``bump'' escapes to spatial infinity in the case when $|x_{k}|\to\ii$. 

The following result provides a useful link between $\mass (\bm{u} )$ and subcritical $L^p$ norms. For $s=1$, this is a classical result due to Lions; see \cite[Lemma I.1]{Lions-84b}. The generalization to $0 < s <1$ can be found in \cite[Lemma 2.3]{Lopes-00} and \cite[Lemma 7.2]{LeLe10}.

\begin{lemma}[A subcritical estimate involving $\mass (\bm{u})$]\label{lem:subcritical_bound}
For every bounded sequence $\bm{u}=\{u_n\}$ in $H^s(\R^d)$, it holds 
\begin{equation}
\limsup_{n\to\ii}\int_{\R^d}|u_n|^{2+\tfrac{4s}d}\leq C\;\mass(\bm{u})^{\tfrac{2s}{d}}\;\limsup_{n\to\ii}\norm{u_n}_{H^s(\R^d)}^2
\label{eq:subcritical_bound_m}
\end{equation}
where $C$ only depends on $0<s\leq1$ and $d\geq1$.
\end{lemma}

\begin{proof}
Here we follow \cite{Lions-84a} and \cite{LeLe10}. We pick a smooth partition of unity $\{\chi_z\}_{z\in\Z^d}$ such that $\sum_{k \in \Z} \chi_z^2 \equiv 1$ on $\R^d$, where $\chi_z$ has its support in the cube $C_z':=\prod_{j=1}^d[z_j-1,z_j+2)$ and equals 1 on the cube $C_z:=\prod_{j=1}^d[z_j,z_j+1)$. We also assume that $\sup_{z\in\Z^d}\norm{\nabla\chi_z}$ is finite, which obviously implies that $\sum_{z\in\Z^d}|\nabla\chi_z|^2\in L^\ii(\R^d)$. We now choose $2<q<\ii$ and $0<\theta<1$ such that $1/q=\theta/2+(1-\theta)/p^*$ and $(1-\theta)q=2$, which yields that $q=2+{4s}/d$. By H\"older's and Sobolev's inequality, we conclude
\begin{align*}
\int_{\R^d}|u|^q=\sum_{z\in\Z^d}\int_{\R^d}\chi_k^2|u|^q&\leq \sum_{z\in\Z^d}\norm{\1_{C'_z}u}_{L^2(\R^d)}^{\theta q}\norm{\chi_ku}_{L^{p^*}(\R^d)}^{(1-\theta)q}\\
&\leq C \sup_{z\in\Z^d}\norm{\1_{C'_z}u}_{L^2(\R^d)}^{\theta q}\pscal{u,\left(\sum_{z\in\Z^d}\chi_k(1-\Delta)^{s}\chi_k\right)u}
\end{align*}
for any $u\in H^{s}(\R^d)$.
Using the localization formula~\eqref{eq:frac_IMS}, we obtain the estimate
\begin{equation}
\int_{\R^d}|u|^{2+\tfrac{4s}d}\leq C\left(\sup_{z\in\Z^d}\int_{C_z'}|u|^2\right)^{\tfrac{2s}{d}}\norm{u}_{H^s(\R^d)}^2.
\label{eq:estim_vanishing}
\end{equation}
This yields~\eqref{eq:subcritical_bound_m}, using~\eqref{eq:formula_m_Lions}.
\end{proof}

The following consequence of Lemma~\ref{lem:subcritical_bound} allows to characterize when $\mass (\bm{u})=0$ holds, which is equivalent to saying that $u_n(\cdot-x_n)\wto0$ for all $\{x_n\}_{n=1}^\infty \subset\R^d$. 

\begin{lemma}[Vanishing]\label{lem:vanishing}
Let $0 < s \leq 1$. Suppose that $\bm{u}=\{u_n\}_{n=1}^\infty$ is a bounded sequence in $H^s(\R^d)$. Then the following statements are equivalent.

\begin{enumerate}

\item[$(i)$] $\mass (\bm{u})=0$;

\smallskip

\item[$(ii)$]  $\displaystyle\lim_{n\to\ii}\sup_{y\in\R^d}\int_{|x-y|\leq R}|u_n(x)|^2 \,dx=0$ for all $R > 0$;

\smallskip

\item[$(iii)$] $u_n\to0$ strongly in $L^p(\R^d)$ for all $2<p<p^*$, where
$p^*=2d/(d-2s)$ if $d>2s$ and $p^*=+\ii$ otherwise.
\end{enumerate}
\end{lemma}

\begin{proof}
The equivalence $(i) \Leftrightarrow (ii)$ follows from~\eqref{eq:formula_m_Lions} and it is obvious that $(iii)$ implies $(i)$, since $u_n \to 0$ strongly in $L^p(\R^d)$ implies that $u_n(\cdot-x_n) \to 0$ strongly in $L^p(\R^d)$ for any sequence $\{x_n\}_{n=1}^\infty \subset\R^d$. The implication $(ii) \Rightarrow (iii)$ follows from Lemma~\ref{lem:subcritical_bound}: We have $u_n\to0$ in $L^{2+4s/d}(\R^d)$, hence in $L^p(\R^d)$ for all $2<p<p^*$ by the Sobolev inequality and by interpolation.
\end{proof}

The following lemma provides an easy way to extract the locally compact part of a weakly convergent sequence $\bm{u} = \{ u_n \}_{n=1}^\infty$ in $H^s(\R^d)$ and to isolate it from the rest.

\begin{lemma}[Extracting the locally convergent part]\label{lem:dichotomy}
Let $0 < s \leq 1$. Suppose that $\bm{u} = \{u_n\}_{n=1}^\infty \subset H^s(\R^d)$ satisfies $u_n\wto u$ weakly in $H^s(\R^d)$. Consider any sequence of positive real numbers $R_k\to\ii$.
Then there exists two bounded sequences $\bm{u'} = \{u'_k\}$ and $\bm{r}=\{r_k\}$ in $H^s(\R^d)$ and a subsequence $\{u_{n_k}\}$ such that:
\begin{enumerate}
 \item[$(i)$] $\mathrm{supp}(u'_k) \subset B(0,R_k)$ for all $k \in \N$ and $u'_k \wto u$ weakly in $H^s(\R^d)$ and $$
 \mbox{$u_k' \to u$ strongly in $L^p(\R^d)$}
 $$
 for all $2\leq p<p^*$;

\smallskip

\item[$(ii)$] $\mathrm{supp}(r_k)  \subset \R^d \setminus B(0,2R_k)$ and 
$$\mass (\bm{r} ) \leq \mass (\{u_{n_k} \}_{k=1}^\infty)\leq \mass (\bm{u} );$$

\smallskip

\item[$(iii)$] $\lim_{k\to\ii}\norm{u_{n_k}-u'_k-r_k}_{H^s(\R^d)}=0$;

\smallskip

\item[$(iv)$] $\displaystyle\lim_{k\to\ii}\bigg(\pscal{u_{n_k},(a-\Delta)^{s}u_{n_k}}-\pscal{u'_{k},(a-\Delta)^{s}u'_{k}}-\pscal{r_{k},(a-\Delta)^{s}r_{k}}\bigg)=0$ for any $a\geq0$.
\end{enumerate}
\end{lemma}

Our assertion $(iii)$ together with the support properties of $u'_k$ and $r_k$ imply that $\1_{B(0,R_k)}u_k\to u$, that $\1_{B(0,2R_k)\setminus B(0,R_k)}u_k\to 0$, and that $\1_{\R^d\setminus B(0,2R_k)}u_k-r_k\to 0$ strongly in $L^2(\R^d)$. Also, we have
\begin{equation}
\lim_{k\to\ii}\left(\int_{\R^d}|u_{n_k}|^2-\int_{\R^d}|r_{k}|^2\right)=\int_{\R^d}|u'|^2.
\label{eq:limit_mass} 
\end{equation}

\begin{proof}[Proof of Lemma~\ref{lem:dichotomy}] 
Using two concentration functions as in \cite{Lions-84a}, we find a subsequence $\{u_{n_k}\}$ such that
\begin{equation}
\lim_{k\to\ii}\int_{B(0,R_k/2)}|u_{n_k}|^2=\int_{\R^d}|u|^2
\label{eq:concentration1} 
\end{equation}
\begin{equation}
\lim_{k\to\ii}\int_{B(0,4R_k)\setminus B(0,R_k/2)}|u_{n_k}|^2+|(1-\Delta)^{\frac{s}{2}} u_{n_k}|^2=0.
\label{eq:concentration2} 
\end{equation}
Let $0\leq\chi\leq1$ be a smooth localization function which is equal to 1 on the ball $B(0,1/2)$ and vanishes outside the ball $B(0,1)$. Similarly, let $0\leq\eta\leq1$ be a smooth function which is equal to 1 outside the ball $B(0,4)$ and vanishes in $B(0,2)$. Define then $\chi_R:=\chi(\cdot/R)$ and $\eta_R:=\eta(\cdot/R)$ for $R >0$. 

We now introduce $u'_k:=\chi_{R_k}u_{n_k}$ and $r_k:=\eta_{R_k}u_{n_k}$. Using~\eqref{eq:commutator_bound}, we easily see that $u'_k$ and $r_k$ are bounded in $H^s(\R^d)$, and that $u'_k\to u$ strongly in $L^2(\R^d)$ and weakly in $H^s(\R^d)$. The strong convergence in $L^2(\R^d)$ follows from \eqref{eq:concentration1}. We have
\begin{align*}
\norm{u_{n_k}-u'_k-r_k}_{H^{s}(\R^d)}&= \norm{(1-\Delta)^{s/2}(1-\chi_{R_k}-\eta_{R_k})u_{n_k}}_{L^2(\R^d)}\\
&\leq \norm{(1-\chi_{R_k}-\eta_{R_k})(1-\Delta)^{s/2}u_{n_k}}_{L^2(\R^d)}\\
&\qquad\qquad\qquad+\norm{\big[(1-\Delta)^{s/2},1-\chi_{R_k}-\eta_{R_k}\big]u_{n_k}}_{L^2(\R^d)}\\
&\leq \norm{(1-\chi_{R_k}-\eta_{R_k})(1-\Delta)^{s/2}u_{n_k}}_{L^2(\R^d)}+\frac{C}{(R_k)^2}\norm{u_{n_k}}_{L^2(\R^d)} 
\end{align*}
where in the last line we have used the commutator bound~\eqref{eq:commutator_bound}.
The function $1-\chi_{R_k}-\eta_{R_k}$ has its support in the annulus $B(0,4R_k)\setminus B(0,R_k/2)$, and hence we get $u_{n_k}-u'_k-r_k\to0$ in $H^{s}(\R^d)$, by \eqref{eq:concentration2}. 

Suppose now that $f(x)$ is a sufficiently smooth function on $\R^d$. For $a>0$, it follows from~\eqref{eq:commutator_bound} that
\begin{equation}
\norm{\left[f_R,(a-\Delta)^s\right]}=O(R^{-1})
\label{eq:commutator_bound_scaling1}
\end{equation}
with, as before, $f_R(x)=f(x/R)$ with $R>0$. In the case when $a=0$, we have, by Remark \ref{rem:comm} and scaling, the following estimate
\begin{equation}
\norm{\left[f_R,(-\Delta)^s\right]}=R^{-s}\norm{\left[f,(-\Delta)^s\right]}=O(R^{-s}).
\label{eq:commutator_bound_scaling2}
\end{equation}
Similarly, we get
\begin{equation}
\norm{\left[f_R,\left[f_R,(-\Delta)^s\right]\right]}=O(R^{-2s}).
\label{eq:doublecommutator_bound_scaling2}
\end{equation}
Using now \eqref{eq:localization} when $a>0$ and~\eqref{eq:doublecommutator_bound_scaling2} when $a=0$, we find that
\begin{multline*}
\lim_{k\to0}\bigg\|(a-\Delta)^s-\chi_{R_k}(a-\Delta)^s\chi_{R_k}-\eta_{R_k}(a-\Delta)^s\eta_{R_k}\\
-\sqrt{1-\chi_{R_k}^2-\eta_{R_k}^2}(a-\Delta)^s\sqrt{1-\chi_{R_k}^2-\eta_{R_k}^2}\bigg\|=0.
\end{multline*}
Therefore,
\begin{multline*}
\pscal{u_{n_k},(a-\Delta)^{s}u_{n_k}}-\pscal{u'_{k},(a-\Delta)^{s}u'_{k}}-\pscal{r_{k},(a-\Delta)^{s}r_{k}}\\
=\norm{(a-\Delta)^{s/2}\sqrt{1-\chi_{R_k}^2-\eta_{R_k}^2}u_{n_k}}^2+o(1).
\end{multline*}
Next, we note that
\begin{equation*}
\norm{(a-\Delta)^{s/2}\sqrt{1-\chi_{R_k}^2-\eta_{R_k}^2}u_{n_k}}^2\leq \max(a,1)^s\norm{(1-\Delta)^{s/2}\sqrt{1-\chi_{R_k}^2-\eta_{R_k}^2}u_{n_k}}^2.
\end{equation*}
The term on the right side tends to zero by \eqref{eq:concentration2}, since $1-\chi_{R_k}^2-\eta_{R_k}^2$ has its support in the annulus $B(0,4R_k)\setminus B(0,R_k/2)$, and by the commutator estimate \eqref{eq:commutator_bound_scaling1}. The proof of Lemma~\ref{lem:dichotomy} is now complete. \end{proof}

From Lemma \ref{lem:dichotomy}, we deduce the following decomposition result for bounded sequences in $H^s(\R^d)$.

\begin{lemma}[Splitting of bounded sequences  in $H^s(\R^d)$]\label{lem:splitting}
Let $0 < s \leq 1$. Suppose that $\bm{u} = \{u_n\}_{n=1}^\infty$ is a bounded sequence in $H^s(\R^d)$. Then there exists a family of sequences $\{ \bm{u}^j \}_{j\geq 1} =  \{ u^j_k \}_{j \geq 1, k \geq 1}$ in $H^s(\R^d)$ such that the following holds. For any fixed $\eps>0$, the following exist:
\begin{itemize}
\item an integer $J\geq0$;
\item $J$ bounded sequences $\bm{u}^j=\{u^j_k\}_{k\geq1}$ in $H^s(\R^d)$ with $j=1,...,J$;
\item a bounded sequence $\bm{r}^J = \{r_k^{J}\}_{k\geq1}$ in $H^s(\R^d)$;
\item $J$ sequences of vectors $\{x^j_k\}_{k\geq1}$ in $R^d$, with $j=1,...,J$;
\item a sequence of positive real numbers $R_k\to +\ii$;
\end{itemize}
and a subsequence $\{u_{n_k}\}$ of $\bm{u} = \{ u_n \}_{n=1}^\infty$ such that:

\begin{enumerate}
 \item [$(i)$] For every $j=1,...,J$ and $k\geq1$, we have $\mathrm{supp}(u^j_k) \subset B(0,R_k)$ and  $u^j_k \weakto u^j$ in $H^s(\R^d)$ and $u^j_k \to u^j$ strongly in $L^p(\R^d)$ for all $2\leq p<p^*$;

\smallskip

\item[$(ii)$] For every $j=1,...,J$ and $k\geq1$, we have $\mathrm{supp} ( r_k^{J}) \subset \R^d\setminus\cup_{j=1}^JB(x_k^j,2R_k)$ and
$$\mass ( \bm{r}^J ) \leq \eps;$$

\smallskip

\item[$(iii)$] $\displaystyle\lim_{k\to\ii}\norm{u_{n_k}-\sum_{j=1}^Ju^j_k(\cdot-x_k^j)-r_k^J}_{H^s(\R^d)}=0$;

\smallskip

\item[$(iv)$] $\displaystyle\lim_{k\to\ii}\bigg(\pscal{u_{n_k},(a-\Delta)^{s}u_{n_k}}-\sum_{j=1}^J\pscal{u^j_{k},(a-\Delta)^{s}u^j_{k}}-\pscal{r^J_{k},(a-\Delta)^{s}r_{k}^J}\bigg)=0$ for any $a\geq0$;

\smallskip

\item[$(v)$] $\displaystyle|x_k^j-x_k^{j'}|\geq 5R_k$ for all $k\geq1$ and $j\neq j'$.
\end{enumerate}
\end{lemma}

Note that in $(iv)$ one can use the strong convergence of $u^j_k\to u^j$ to replace $u^j_k(\cdot-x_k^j)$ with $u^j(\cdot-x_k^j)$ in $(iv)$, at the expense of replacing the $H^s(\R^d)$ norm by any $L^p(\R^d)$ norm with $2\leq p<p^*$. Similarly, we can use Fatou's lemma to replace $(iv)$ with
$$\liminf_{k\to\ii}\bigg\{\pscal{u_{n_k},(a-\Delta)^{s}u_{n_k}} - \pscal{r^J_{k},(a-\Delta)^{s}r_{k}^J}\bigg\}\geq \sum_{j=1}^J\pscal{u^j,(a-\Delta)^{s}u^j}.$$
It is therefore possible to state a weaker result, which does not involve the functions $u_k^j$. However, strong convergence in $H^s(\R^d)$ is sometimes useful in practice and we prefer the previous version.

Results in the spirit of Lemma \ref{lem:splitting} are ubiquitous in the literature (see, e.g., \cite{Struwe-84,BreCor-85,Lions-87,Gerard-98} and~\cite{BahGer-99,MerVeg-98,Keraani-01} in the time-dependent case). The proof is based on Lemma \ref{lem:dichotomy} and it proceeds induction.

Let us briefly sketch the arguments here. If $\mass (\bm{u}) = 0$, there is nothing to prove by Lemma~\ref{lem:vanishing}. If $\mass (\bm{u}))>0$, we choose a first limit $u^1$ up to translations and extraction of a subsequence (note that the mass of $u^1$ can be chosen as close to $\mass(\bm{u})$ as needed). Next we use Lemma~\ref{lem:dichotomy} to construct $u^1_k$ and $r^1_k$. Then we apply the whole process again to the sequence $\bm{r}^1 = \{r^1_k\}_{k=1}^\infty$, and so on. At each step, we can choose the limit $u^{J+1}$ such as to have, for instance,
\begin{equation}
\frac{\mass(\bm{r^J})}{2} \leq\int_{\R^d}|u^{J+1}|^2\leq \mass(\bm{r}^J).
\label{eq:choice_mass} 
\end{equation}
The family of sequences given by $\bm{u}^j = \{u^j\}_{j\geq 1}$ satisfies, for every $J$, the relation
$$\lim_{k\to\ii}\left(\int_{\R^d}|u_{n_k}|^2-\int_{\R^d}|r_{k}^J|^2\right)=\sum_{j=1}^J\int_{\R^d}|u^j|^2,$$
similarly to~\eqref{eq:limit_mass}. In particular, this implies that
$$\sum_{j=1}^J\int_{\R^d}|u^j|^2\leq \limsup_{n\to\ii}\int_{\R^d}|u_n|^2 < +\infty.$$
Thus the series $\sum_{j}\int_{\R^d}|u^j|^2$ must be convergent, and hence $\lim_{j\to\ii}\int_{\R^d}|u^j|^2=0$. By~\eqref{eq:choice_mass}, this proves the assertion that $\mass(\bm{r}^J)$ can be made arbitrarily small. The other technical details are left to the reader.

\bibliographystyle{siam}
\bibliography{MyBib}

\end{document}